\documentclass{article}

\usepackage[a4paper]{geometry}
\usepackage[utf8]{inputenc}
\usepackage[T1]{fontenc}
\usepackage{amsmath}
\usepackage{amssymb}
\usepackage{amsfonts}
\usepackage{amsthm}
\usepackage{mathtools}
\usepackage{theoremref}
\usepackage{graphicx}
\usepackage[justification=centering]{caption}
\usepackage{subcaption}
\usepackage[french, onelanguage, ruled, vlined]{algorithm2e}
\usepackage{pgfplots}
\usepackage[section]{placeins}
\usepackage{stackengine,scalerel}
\usepackage{titling}

\thanksmarkseries{arabic}

\newcommand\blfootnote[1]{%
  \begingroup
  \renewcommand\thefootnote{}\footnote{#1}%
  \addtocounter{footnote}{-1}%
  \endgroup
}

\graphicspath{{figures/}}

\theoremstyle{definition}
\newtheorem{definition}{Definition}[section]

\theoremstyle{plain}

\theoremstyle{plain}

\theoremstyle{plain}

\theoremstyle{plain}
\newtheorem{proposition}[definition]{Proposition}

\theoremstyle{plain}
\newtheorem{lemma}[definition]{Lemma}

\theoremstyle{remark}
\newtheorem*{remark}{Remark}

\usepackage{tikz}
\usepackage{subcaption}
\usepackage{forest}
\usepackage{pgfplots}
\usepackage{textcomp}
\usepackage[inline]{asymptote}
\usepackage{tikz-cd}
\usetikzlibrary{decorations.pathreplacing,decorations.markings}

\usetikzlibrary{positioning}
\usetikzlibrary{arrows.meta}
\tikzset{|/.tip={Bar[width=.8ex,round]}}

\usetikzlibrary{calc}

\tikzset{vertex/.style={circle,draw=black,fill=white,inner sep=1,minimum size=8pt}}
\tikzset{leg/.style={rectangle,sharp corners,scale=0.8,draw=black,fill=blue!25,inner sep=1,minimum size=12pt}}
\tikzset{subtree/.style={rectangle,draw=black,fill=white,minimum size=16pt}}
\tikzset{bits/.style={rectangle,fill=black!25,minimum size=12pt}}
\tikzset{every loop/.style={}}
\tikzset{marknode/.style={rectangle,sharp corners,fill=black,minimum size=4pt}}
\tikzset{label node/.style={rectangle,rounded corners,draw=black,fill=white}}
\tikzset{edge label/.style={fill=white, inner sep=0.2mm}}
\tikzset{halfedge label/.style={fill=white, inner sep=0.2mm, scale=0.8}}
\tikzset{small halfedge label/.style={fill=white, inner sep=0.2mm, scale=0.5}}

\tikzset{node/.style args={#1}{circle,draw=black,fill=#1,inner sep=1,minimum size=5pt}}
\tikzset{star/.style={fill=white, circle, inner sep=0.1}}

\tikzset{
  mid arrow/.style args={#1}{postaction={decorate,decoration={
        markings,
        mark=at position #1 with {\arrow{>}}
      }}},
  mid arrow/.default={0.5}
}

\tikzset{
    partial ellipse/.style args={#1:#2:#3}{
        insert path={+ (#1:#3) arc (#1:#2:#3)}
    }
}

\tikzset{
    partial circle/.style args={#1:#2:#3}{
        insert path={+ (#1:#3) arc (#1:#2:#3)}
    }
}

\tikzset{
  pics/fancode/.style args={#1,#2/#3;#4}{
    code={
      \coordinate (center) at (0, 0);
      \foreach \i in {#1,...,#2}{
        \begin{scope}[rotate={180+(\i-1)*360/#3}, transform shape, name prefix=\i]
          #4
        \end{scope}
      }
    }
  }
}

\tikzset{
  pics/blankpic/.style={
    code={
    }
  }
}

\tikzset{
  pics/nodepic/.style args={#1}{
    code={
      \node[node=#1] (node) at (0, 0) {};
    }
  }
}

\tikzset{
  pics/fanpic/.style args={#1-#2/#3;#4;#5}{
    code={
      \coordinate (center) at (0, 0);
      \foreach \i in {#1,...,#2}{
        \begin{scope}[rotate={180+\i*360/#3}, shift={(#4, 0)}]
          \pic (\i) at (0, 0) {#5};
        \end{scope}
      }
    }
  }
}

\usepackage{amsmath}
\usepackage{amssymb}
\usepackage{amsfonts}
\usepackage{amsthm}
\usepackage{mathtools}
\usepackage{theoremref}
\usepackage{graphicx}
\usepackage{subcaption}
\usepackage{faktor}
\usepackage{ifthen}
\usepackage{multirow}
\usepackage{float}
\restylefloat{table}
\usepackage[normalem]{ulem}

\newcommand{\ignore}[1]{}

\DeclareMathOperator*{\id}{id}

\DeclareMathOperator*{\Gal}{Gal}

\DeclareMathOperator*{\Crit}{Crit}

\DeclareMathOperator*{\Cart}{Cart}
\DeclareMathOperator*{\ord}{ord}
\DeclareMathOperator*{\odd}{odd}
\DeclareMathOperator*{\even}{even}

\newcommand\NN{\mathbb{N}}
\newcommand\ZZ{\mathbb{Z}}
\newcommand\ZZn[1]{\ZZ/#1\ZZ}
\newcommand\QQ{\mathbb{Q}}
\newcommand\QQb{\bar{\mathbb{Q}}}
\newcommand\CC{\mathbb{C}}
\newcommand\RR{\mathbb{R}}
\newcommand\PP{\mathbb{P}^1}

\newcommand\cD{\mathcal{D}}

\newcommand\cM{\mathcal{M}}
\newcommand\Sym{\mathfrak{S}}
\newcommand\AGG{\Gal(\QQb/\QQ)}
\newcommand\genx{\xi}
\newcommand\geny{\eta}

\newcommand\treep{\oplus}
\newcommand\treem{\ominus}

\title{Regular dessins with moduli fields of the form $\QQ(\zeta_p,\sqrt[p]{q})$}

\author{Nicolas Daire\thanks{Department of Mathematics, \'{E}cole Normale Supérieure, 45 rue d'Ulm, 75005 Paris, France,\newline (e-mail: \href{mailto:nicolas.daire@ens.psl.eu}{nicolas.daire@ens.psl.eu})}, Fumiharu Kato\thanks{Department of Mathematics, Tokyo Institute of Technology, 2-12-1 Ookayama, Meguro, Tokyo 152-8551, Japan, (e-mail: \href{mailto:bungen@math.titech.ac.jp}{bungen@math.titech.ac.jp})}, Yoshiaki Uchino\thanks{Department of Mathematics, Tokyo Institute of Technology, 2-12-1 Ookayama, Meguro, Tokyo 152-8551, Japan, (e-mail: \href{mailto:uchino.y.ab@m.titech.ac.jp}{uchino.y.ab@m.titech.ac.jp})}}
\date{}

\begin{document}

\maketitle

\blfootnote{2020 \textit{Mathematics Subject Classification}. \textup{Primary: 14H57, Secondary: 14H30}.}

\abstract{
  Gareth Jones asked during the 2014 SIGMAP conference for examples of regular dessins with nonabelian fields of moduli. In this paper, we first construct dessins whose moduli fields are nonabelian Galois extensions of the form $\QQ(\zeta_p,\sqrt[p]{q})$, where $p$ is an odd prime and $\zeta_p$ is a $p$th root of unity and $q\in\QQ$ is not a $p$th power, and we then show that their regular closures have the same moduli fields. Finally, in the special case $p=q=3$ we give another example of a regular dessin with moduli field $\QQ(\zeta_3,\sqrt[3]{3})$ of degree $2^{19}\cdot3^4$ and genus $14155777$.
}


\section{Introduction}

Grothendieck first coined the term {\em Dessin d'enfant} in {\em Esquisse d'un Programme} \cite{esquisse} to denote a connected bicolored graph embedded on a compact connected oriented topological surface. The study was motivated by the one to one correspondance between dessins d'enfant, the combinatorial data of the associated  cartographical group, and the geometric concept of coverings of $\PP$ by compact Riemann surfaces ramified at most over three points. Moreover, by Belyi's theorem any such covering is given the structure of an algebraic curve defined over a number field, therefore we obtain a natural action of the absolute Galois group $\AGG$ on the set of isomorphism classes of dessins. A lot of the interest for dessins stems from the fact that this action is faithful, providing a way to study the absolute Galois group through its action on the set of dessins. A particularly interesting family of dessins is that of regular dessins, characterized by the fact that their automorphism groups act transitively on their sets of edges, and the Galois action was proved to remain faithful when restricted to the subset of isomorphism classes of regular dessins \cite{Gonzalez-Diez-Jaikin-Zapirain}.

To any dessin we associate a number field called its moduli field, which is defined as the subfield of $\QQb$ fixed by the subgroup of $\AGG$ that fixes the dessin up to isomorphism. Conder, Jones, Streit and Wolfart noted in \cite{CJSW} that the moduli fields of all the examples of regular dessins known at the time were abelian Galois extensions of $\QQ$. Herradón constructed in \cite{Herradon} an explicit equation for a regular dessin whose moduli field $\QQ(\sqrt[3]{2})$ is not a Galois extension of $\QQ$, and Hidalgo later generalized his construction in \cite{Hidalgo} to produce regular dessins whose moduli fields are of the form $\QQ(\sqrt[p]{2})$ where $p$ is an odd prime number. However there is as of yet no known example of regular dessin whose moduli field is a nonabelian Galois extension of $\QQ$. This is the starting point of this paper, in which we will exhibit examples of regular dessins with moduli fields that are nonabelian Galois extensions of $\QQ$.

In the present paper, we begin by recalling the main definitions and results on dessins d'enfant. We will then expose constructions of regular dessins whose moduli fields are nonabelian Galois extensions of $\QQ$. We first exhibit dessins whose moduli fields are of the form $\QQ(\zeta_3,\sqrt[3]{q})$, where $\zeta_3$ is a primitive third root of unity and $q\in\QQ$ is not a third power, and show that the regular closures of these dessins possess the same moduli fields. We then generalize this construction to show that there exist regular dessins with moduli fields $\QQ(\zeta_p,\sqrt[p]{q})$, where $\zeta_p$ is a primitive $p$th root of unity and $q\in\QQ_{>0}$ is not a $p$th power. Finally, we give an example of a regular dessin with moduli field $\QQ(\zeta_3,\sqrt[3]{3})$ of degree 42467328 and genus 14155777.

\paragraph*{Aknowledgements}
The authors are grateful to Professor Jürgen Wolfart for valuable comments.

\bigskip
\paragraph*{Notations}
\begin{itemize}
\item $\Sym_{E}$: the group of self-bijections of the set $E$, similarly $\Sym_n$ is the group of permutations of a set of $n$ elements (we favor a right action, hence we write the product $\sigma\tau\coloneqq\tau\circ\sigma$)
\item $\Gal(E/F)$: the Galois group of $F$-automorphisms of $E$
\item $\zeta_k$: the $k$th primitive root of unity $\exp(\frac{2i\pi}{k})$
\item $F_2$: the free group of rank $2$ with generators $(\genx,\geny)$
\item $\Crit$: the set of critical values of a function
\end{itemize}

\section{Preliminaries on dessins d'enfant}

We refer the reader to existing expositions of the theory such as \cite{Guillot}, \cite{Lando-Zvonkin}, \cite{Jones-Wolfart} and \cite{Girondo-Gonzalez-Diez} for proofs of the presented facts and further details.
\bigskip

A {\em dessin d'enfant} is a connected bipartite graph embedded on a compact connected orientable topological surface, such that the complement of the graph is a disjoint union of 2-cells. Two such dessins are equivalent if there exists an orientation preserving homeomorphism between the underlying surfaces that induces an isomorphism between the embedded bipartite graphs.

A dessin is determined up to isomorphism by a pair $(C,\beta)$ where $C$ is a smooth algebraic curve and $\beta\colon C\to\PP$ is a meromorphic mapping ramified at most over $\{0,1,\infty\}$, and by Belyi's theorem we can further ask for $C$ and $\beta$ to both be defined over a number field. We call $(C,\beta)$ a {\em Belyi pair} and $\beta$ a {\em Belyi function}. The corresponding graph embedding on the underlying surface is recovered by pulling back the segment $[0,1]$ along $\beta$, we define black and white vertices as the preimages of $0$ and $1$ respectively, and the edges as the preimages of $]0,1[$.

By covering theory a dessin is also determined up to isomorphism by the {\em monodromy action} of the fundamental group of the complex projective line $\pi_1(\PP)$ on the fiber over the point $\frac{1}{2}$ which is identified to the set of edges of the dessin. The fundamental group $\pi_1(\PP)$ is isomorphic to the free group of rank two $F_2=\langle\genx,\geny\rangle$ with generators $\genx$ and $\geny$ which are two loops with base point $\frac{1}{2}$ and circling counter-clockwise around $0$ and $1$ respectively. The monodromy action of the generators $\genx$ and $\geny$ then corresponds to the product of the counter-clockwise cyclic permutation of the edges around black and white vertices respectively. We call {\em monodromy map} $M\colon F_2\to\Sym_E$ the map that associates to each element of $F_2$ the corresponding permutation of the set of edges, and we call {\em cartographic group} the image of the monodromy map, which is a transitive subgroup of the group of permutations of the set of edges.

\medskip
When the {\em automorphism group} of a dessin $\cD$ acts transitively on the set of edges, we say that $\cD$ is a {\em regular dessin}. When that is the case the cartographic group $G$ acts transitively and freely on the set of edges, the monodromy action is thus given by the canonical action of $G$ on itself. There is a natural bijection between regular dessins and finite groups generated by two distinguished elements $\genx$ and $\geny$ up to isomorphism. Two regular dessins determined by $G_1=\langle\genx_1,\geny_1\rangle$ and $G_2=\langle\genx_2,\geny_2\rangle$ respectively are isomorphic if and only if there exists an isomorphism between $G_1$ and $G_2$ that preserves the distinguished generators. Given a dessin $\cD$, there exists a unique regular dessin $\widetilde{\cD}$ with a morphism $\phi\colon\widetilde{\cD}\to\cD$ such that any morphism from a regular dessin to $\cD$ factors through $\phi$. We call $\widetilde{\cD}$ the {\em regular closure} of $\cD$. Moreover, there exists an isomorphism $\Cart(\widetilde{\cD})\cong\Cart(\cD)$ that preserves the distinguished generators. There exists a natural action of the absolute Galois group $\AGG$ on the set of isomorphism classes of dessins, we denote by $\cD^\sigma$ the action of an automorphism $\sigma$ on a dessin $\cD$, and this Galois action commutes with regular closure, i.e. we have $(\widetilde{\cD})^\sigma\cong\widetilde{(\cD^\sigma)}$.

\medskip
Given a dessin $\cD$, we say that a number field $k$ is a {\em field of definition} of $\cD$ if $\cD$ is isomorphic to a dessin defined over $k$. However there does not necessarily exist a smallest field of definition. We thus define the {\em moduli field} of a dessin $\cD$ as the subfield of $\QQb$ fixed by the subgroup of $\AGG$ constituted of the elements fixing $\cD$ up to isomorphism. The moduli field of a dessin is contained in all fields of definition but is not necessarily itself a field of definition, however it is the case in particular for regular dessins.

\section{Constructions of regular dessins with nonabelian moduli fields}
We are now ready to give examples of regular dessins whose moduli fields are nonabelian Galois extensions of $\QQ$. To do so, we will first exhibit dessins with such moduli fields, and then prove that their regular closures admit the same moduli fields.

\bigskip
Before proceeding with the examples, let us first present a classic family of Belyi polynomials that we will use in the following constructions. For positive integers $m,n\in\NN$ we define the polynomial $$B_{m,n}\coloneqq\frac{(m+n)^{m+n}}{m^mn^n}X^m(1-X)^n\in\QQ[X].$$ By computing the derivative $B_{m,n}'=\frac{(m+n)^{m+n}}{m^mn^n}X^{m-1}(1-X)^{n-1}(m-(m+n)X)$ we verify that $B_{m,n}:\PP\to\PP$ is a Belyi function that ramifies only at $0$, $1$, $\infty$ and $\frac{m}{m+n}$ with ramification indices $m$, $n$, $m+n$ and $2$ respectively, and $B_{m,n}(0)=0$, $B_{m,n}(1)=0$, $B_{m,n}(\infty)=\infty$ and $B_{m,n}(\frac{m}{m+n})=1$ (see Figure \ref{dessin-Bmn}).

\begin{figure}[H]
  \centering
  \includegraphics[height=2cm]{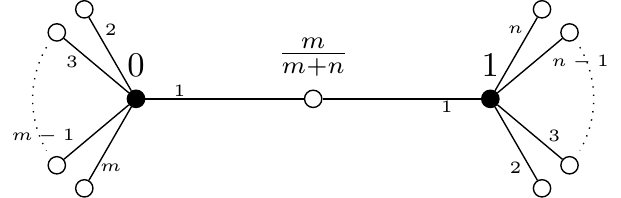}
  \caption{Dessin corresponding to the Belyi pair $(\PP,B_{m,n})$.}
  \label{dessin-Bmn}
\end{figure}

\subsection{Regular dessins with moduli fields of the form $\QQ(\zeta_3,\sqrt[3]{q})$}\label{ex1}

Let $q\in\QQ_{>0}$ be a positive rational number that is not a third power.

\medskip
Let $m,n\in\NN$ be coprime positive integers such that $\frac{27}{27+q^2}=\frac{m}{m+n}$, and let
\begin{align*}
  C&\colon y^2=x(x-(1-\zeta_3))(x-\sqrt[3]{q}),\\
  \beta&\colon C\to\PP,\;(x,y)\mapsto\frac{1}{27^mq^{2n}}(x^6+27)^m(q^2-x^6)^n.
\end{align*}

The function $\beta$ is given by the composition $\beta=\beta_1\circ\beta_0\circ\pi$ of the following maps.
\begin{enumerate}
\item $\pi\colon C\to\PP$ is the projection on the coordinate $x$, which is ramified over $\{0,1-\zeta_3,\sqrt[3]{q},\infty\}$.
\item $\beta_0\coloneqq X^6\in\QQ[X]$, $\Crit(\beta_0)=\{0\}$ so $\beta_1\circ\pi$ ramifies over $\{0,(1-\zeta_3)^6=-27,q^2,\infty\}$.
\item $\beta_1\coloneqq B_{m,n}(\frac{X+27}{q^2+27})$, so $\beta=\beta_1\circ\beta_0\circ\pi$ ramifies over $\{0,1,\infty\}$.
\end{enumerate}

The pair $(C,\beta)$ is thus a Belyi pair, and we call $\cD$ the corresponding dessin. The dessin $\cD$ is defined over $\QQ(\zeta_3,\sqrt[3]{q})$, so its moduli field is a subfield of $\QQ(\zeta_3,\sqrt[3]{q})$. By taking the regular closure we then obtain the inclusion of moduli fields $\cM(\widetilde{D})\subseteq\cM(D)\subseteq\QQ(\zeta_3,\sqrt[3]{q})$, and moreover $\widetilde{D}$ is regular so it is defined over $\cM(\widetilde{D})$. We shall prove that $\cM(\widetilde{\cD})$ is in fact exactly $\QQ(\zeta_3,\sqrt[3]{q})$, which is a nonabelian Galois extension of $\QQ$ with Galois group $$\Gal(\QQ(\zeta_3,\sqrt[3]{q})/\QQ)\cong\Sym_3.$$ To that end we must show that an automorphism $\sigma\in\AGG$ fixes $\widetilde{\cD}$ if and only if it fixes $\zeta_3$ and $\sqrt[3]{q}$, or equivalently that $\Gal(\QQ(\zeta_3,\sqrt[3]{q})/\QQ)$ acts freely on the orbit of $\widetilde{\cD}$.

  \medskip
  Let $\sigma\in\AGG$, the Galois conjugate $\cD^\sigma$ is given by the Belyi pair $(C^\sigma,\beta^\sigma)$, where $$C^\sigma\colon y^2=x(x-(1-\sigma(\zeta_3)))(x-\sigma(\sqrt[3]{q})),$$ and $\beta^\sigma$ has the same expression as $\beta$ because all of its coefficients are rational. The orbit of the pair $(\zeta_3,\sqrt[3]{q})$ by $\AGG$ is $\{\zeta_3^i,\zeta_3^j\sqrt[3]{q}\}_{1\leq i\leq2,0\leq j\leq2}$. Elliptic curves given by equations of the form $y^2=(x-a)(x-b)(x-c)$ are isomorphic if and only if the cross-ratios of the tuples $(a,b,c,\infty)$ coincide. We verify that the cross-ratios are all distinct, so the orbit of $\cD$ is given by the six dessins $\cD^\sigma$ for $\sigma\in\Gal(\QQ(\zeta_3,\sqrt[3]{q})/\QQ)$. As a consequence $\cM(\cD)=\QQ(\zeta_3,\sqrt[3]{q})$. To prove that the regular closures $\widetilde{\cD^\sigma}$ constituting the orbit of $\widetilde{\cD}$ are also non isomorphic, we must first draw the dessins $\cD^\sigma$ to compute their cartographic groups.

\medskip
Let us first draw the dessin $\cD_0$ corresponding to the Belyi pair $(\PP,\beta_1\circ\beta_0)$ (see Figure \ref{beta-ex1}). The dessin $\cD_0$ is defined over $\QQ$, so the dessins $\cD^\sigma$ in the orbit are then obtained by lifting $\cD_0$ to the curves $C^\sigma$. To simplify the graphical representations of the dessins, we will use the notation in Figure \ref{simply} for consectutive edges incident to a vertex.

\begin{figure}[H]
  \centering
  \includegraphics[height=2.5cm]{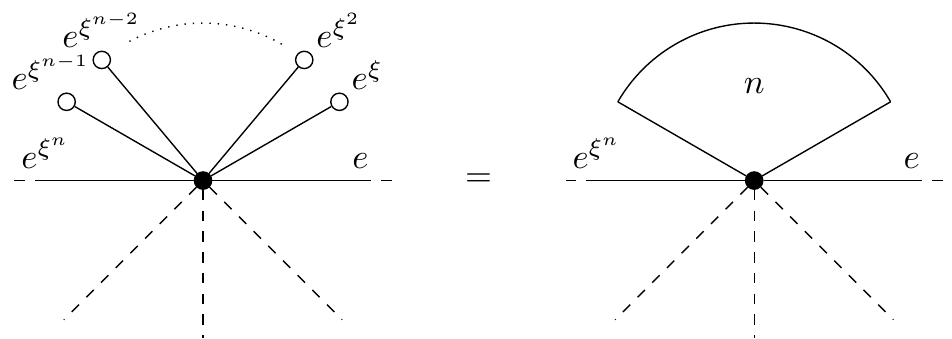}
  \caption{Notation for consecutive edges.}
  \label{simply}
\end{figure}

\begin{figure}[H]
  \centering
  \includegraphics[width=6cm]{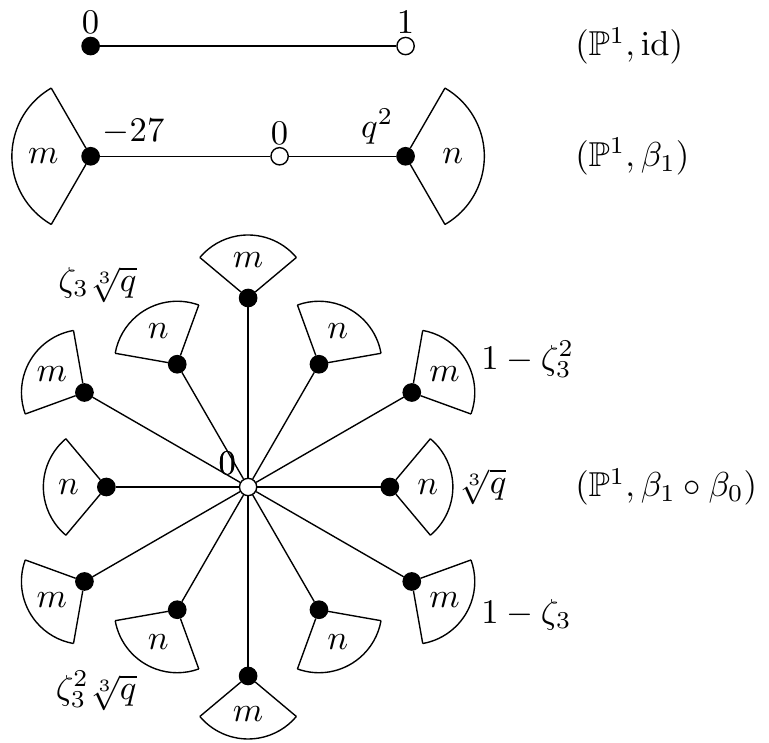}
  \caption{Construction of $\cD_0$.}
  \label{beta-ex1}
\end{figure}

The dessins $\cD_1,\dots,\cD_6$ conjugate to $\cD$ are embedded on a torus, so in the representations in Figure \ref{ex1-dessins} we will identify the outermost edges on opposite sides.

\begin{figure}[H]
 \begin{subfigure}{0.5\hsize}
  \begin{center}
   \includegraphics[height=0.9\hsize]{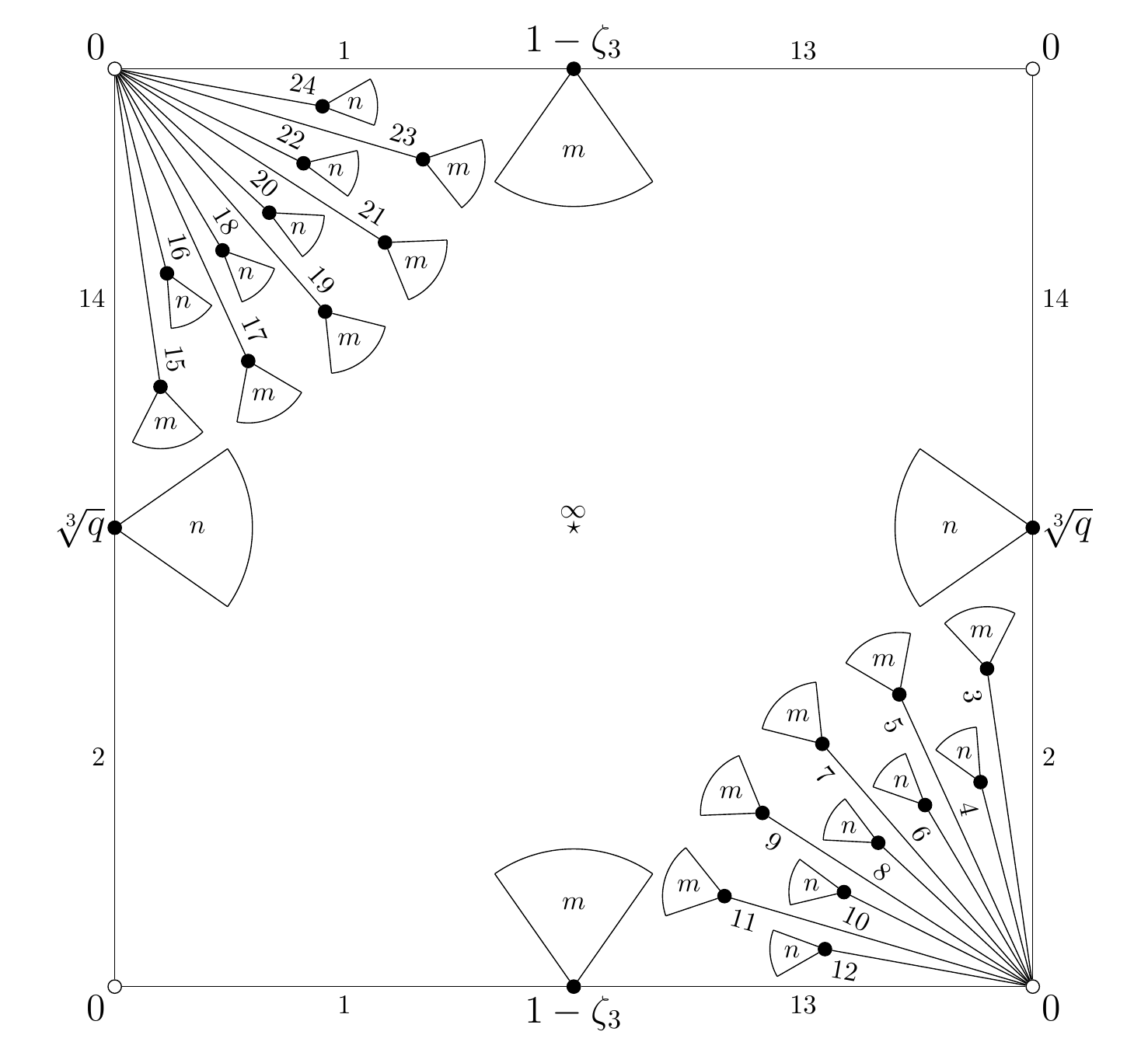}
  \end{center}
  \caption{$\cD_1\coloneqq\cD$}
  \label{ex1-1}
 \end{subfigure}
 \begin{subfigure}{0.5\hsize}
  \begin{center}
   \includegraphics[height=0.9\hsize]{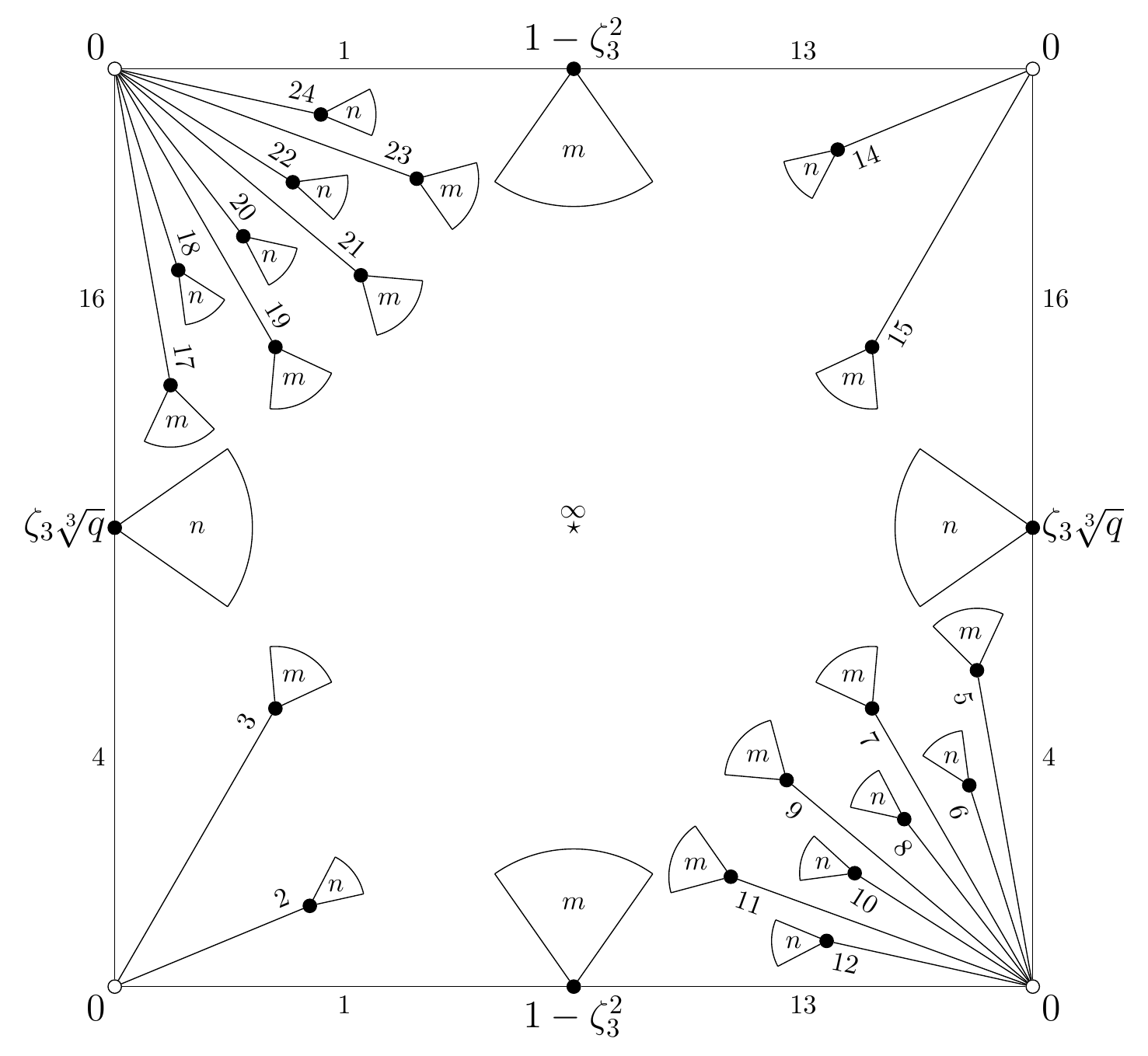}
  \end{center}
  \caption{$\cD_2\coloneqq\cD^\sigma,\sigma\colon(\zeta_3,\sqrt[3]{q})\mapsto(\zeta_3^2,\zeta_3\sqrt[3]{q})$}
 \end{subfigure}
\end{figure}

\begin{figure}[H]\ContinuedFloat
 \begin{subfigure}{0.5\hsize}
  \begin{center}
   \includegraphics[height=0.9\hsize]{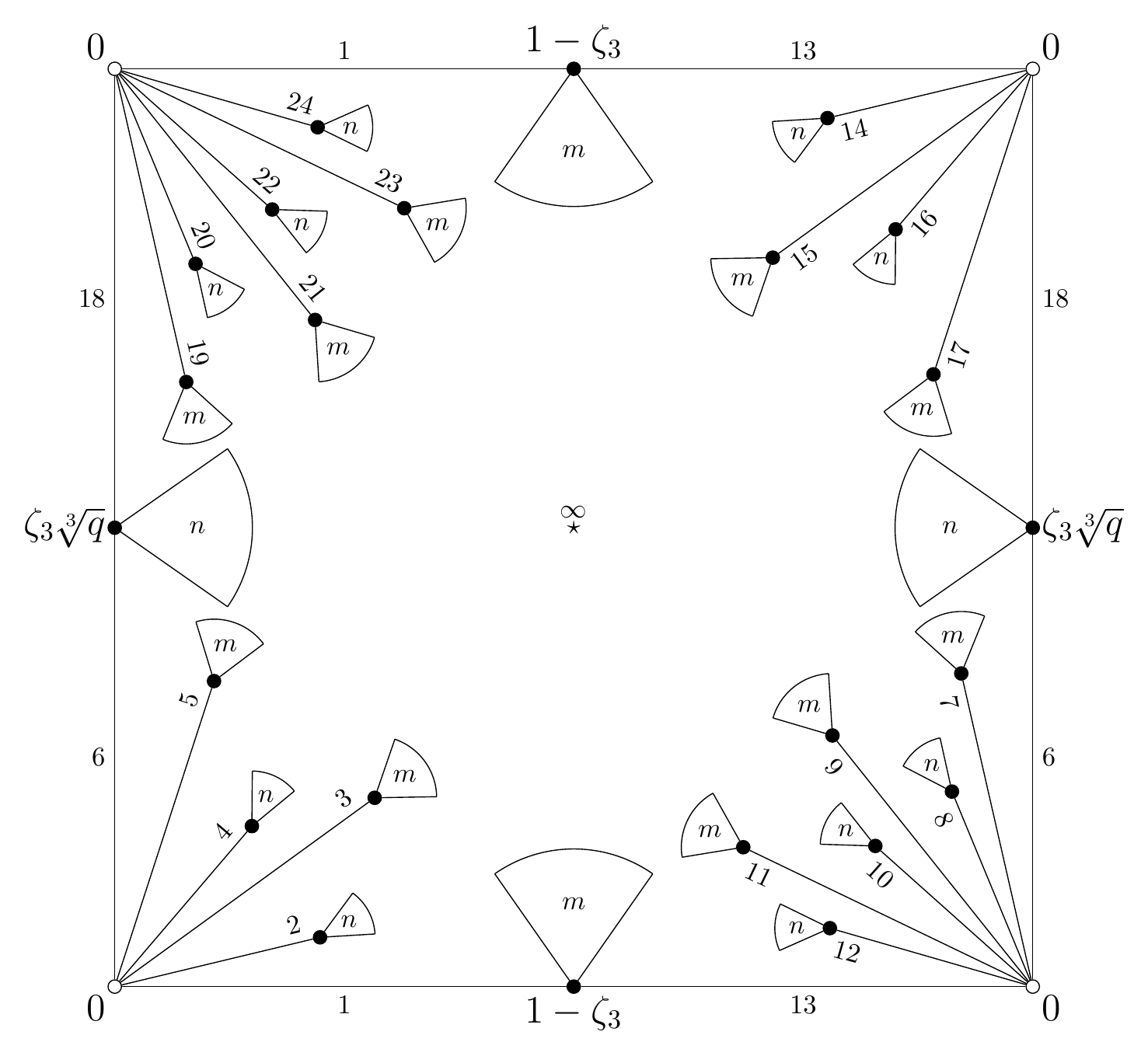}
  \end{center}
  \caption{$\cD_3\coloneqq\cD^\sigma,\sigma\colon(\zeta_3,\sqrt[3]{q})\mapsto(\zeta_3,\zeta_3\sqrt[3]{q})$}
 \end{subfigure}
 \begin{subfigure}{0.5\hsize}
  \begin{center}
   \includegraphics[height=0.9\hsize]{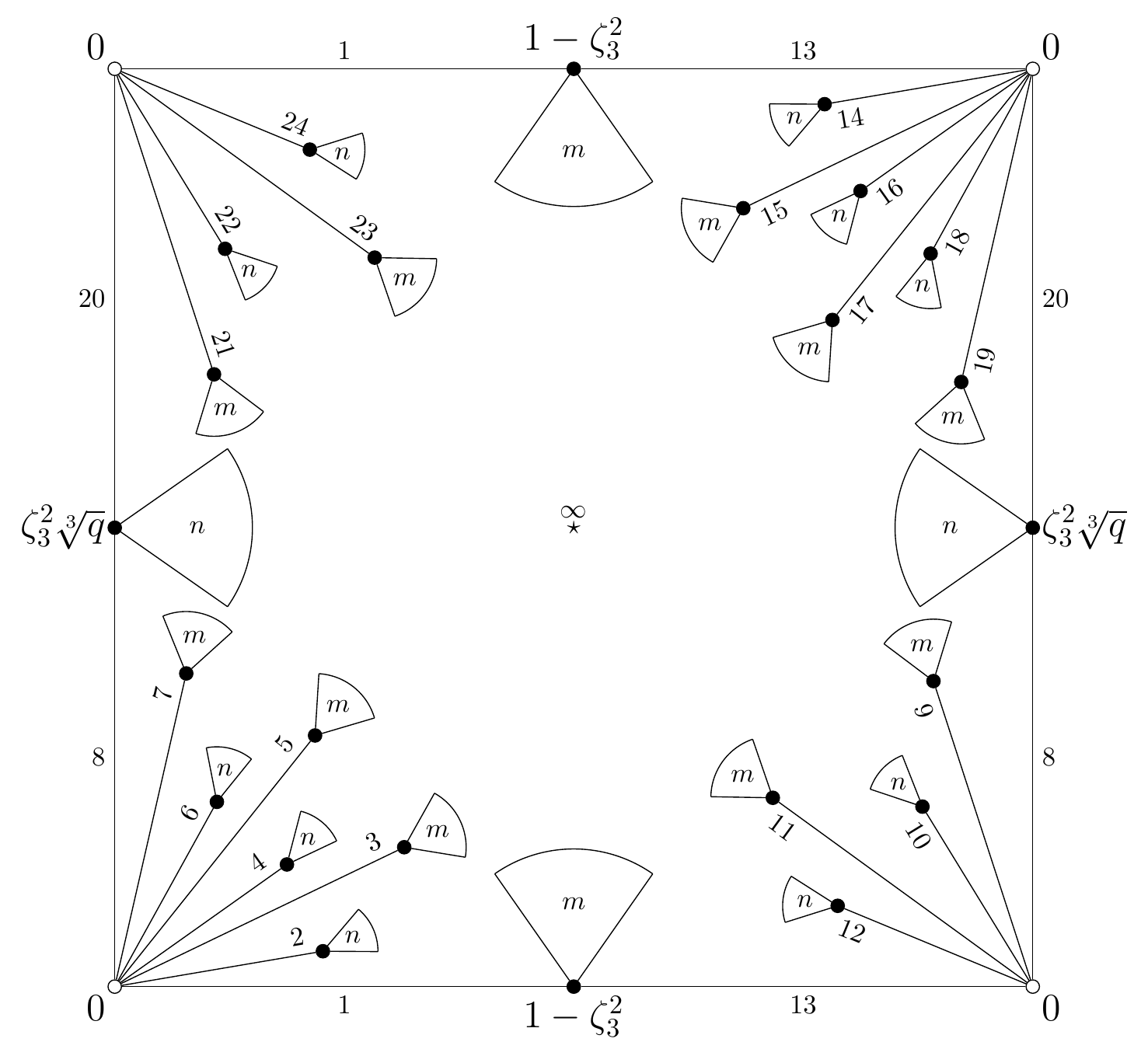}
  \end{center}
  \caption{$\cD_4\coloneqq\cD^\sigma,\sigma\colon(\zeta_3,\sqrt[3]{q})\mapsto(\zeta_3^2,\zeta_3^2\sqrt[3]{q})$}
 \end{subfigure}
\end{figure}

\begin{figure}[H]\ContinuedFloat
 \begin{subfigure}{0.5\hsize}
  \begin{center}
   \includegraphics[height=0.9\hsize]{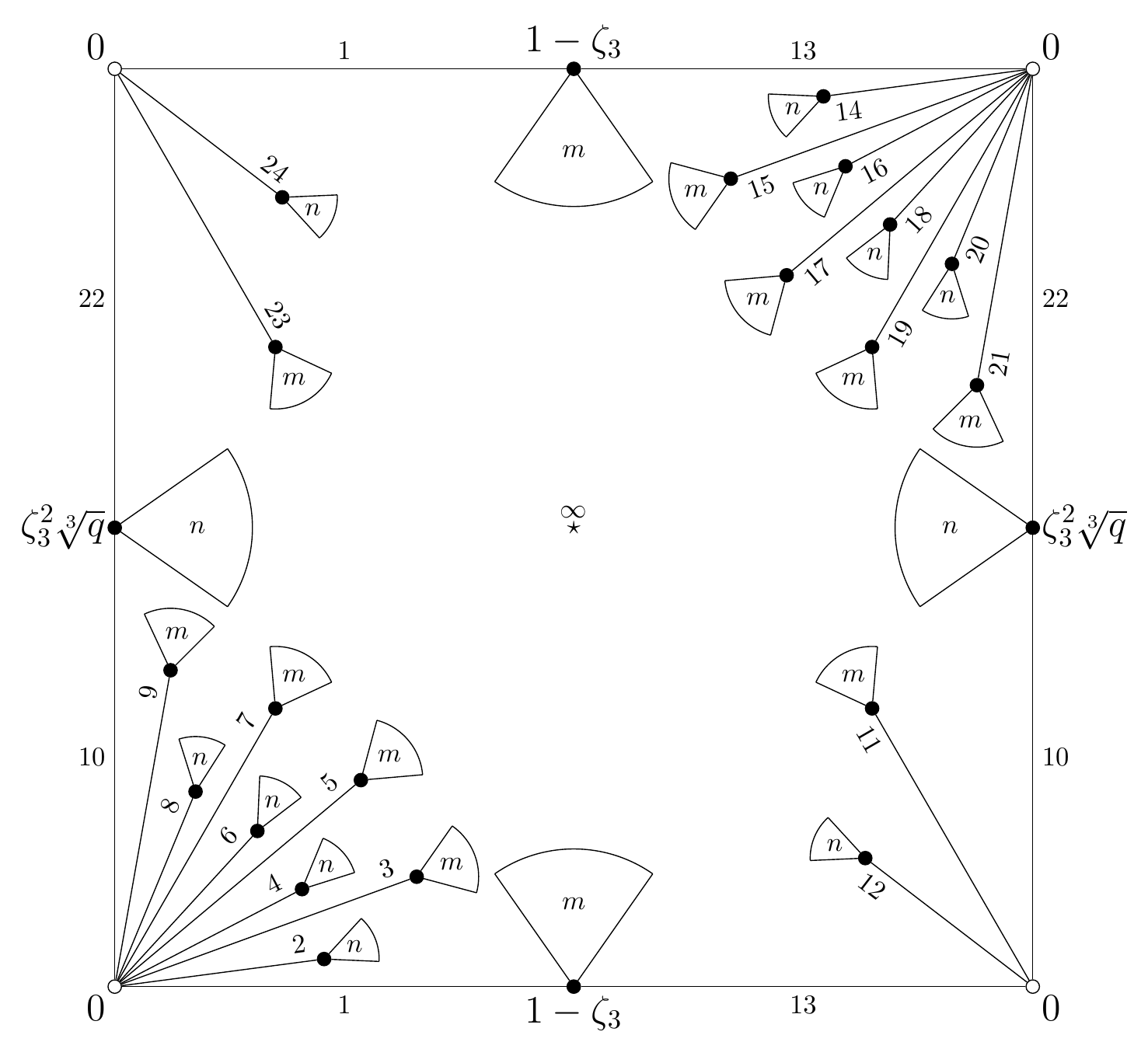}
  \end{center}
  \caption{$\cD_5\coloneqq\cD^\sigma,\sigma\colon(\zeta_3,\sqrt[3]{q})\mapsto(\zeta_3,\zeta_3^2\sqrt[3]{q})$}
 \end{subfigure}
 \begin{subfigure}{0.5\hsize}
  \begin{center}
   \includegraphics[height=0.9\hsize]{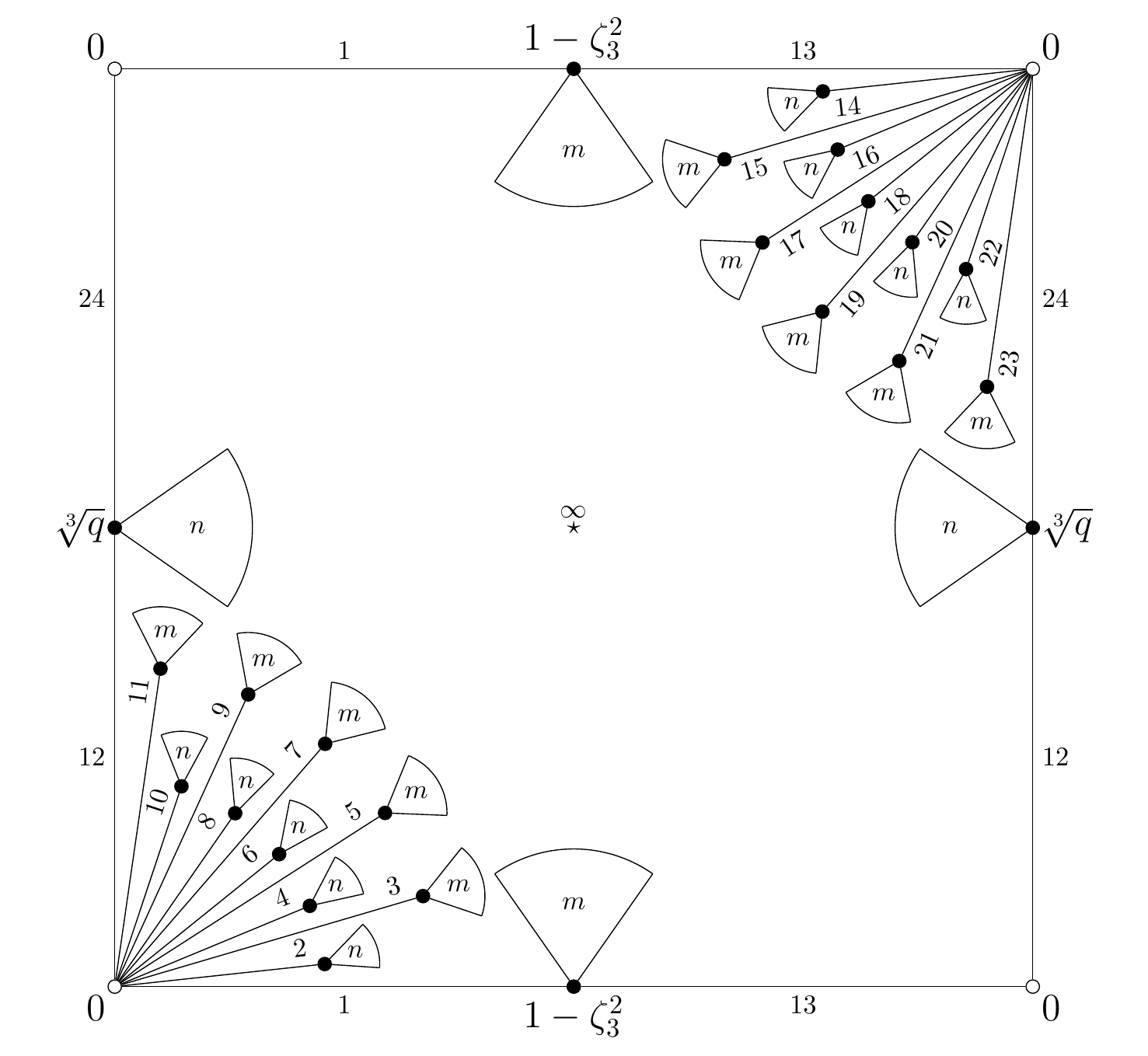}
  \end{center}
  \caption{$\cD_6\coloneqq\cD^\sigma,\sigma\colon(\zeta_3,\sqrt[3]{q})\mapsto(\zeta_3^2,\sqrt[3]{q})$}
  \label{ex1-6}
\end{subfigure}
\caption{Dessins $\cD_1,\dots,\cD_6$ in the Galois orbit of $\cD$.}
\label{ex1-dessins}
\end{figure}

We will now establish that $\widetilde{\cD_1}$ is not isomorphic to $\widetilde{\cD_2},\dots,\widetilde{\cD_6}$. To that end it suffices to show that there is no isomorphism between the cartographic groups fixing the canonical generators. We shall therefore exhibit an element $\omega\in F_2=\langle \genx,\geny\rangle$ such that $M_k(\omega)$ commutes with $M_k(\geny^2)$ only when $k=1$, where $M_k$ is the monodromy map of $\cD_k$.

\medskip
We have defined $m$ and $n$ to be positive coprime integers such that $\frac{27}{27+q^2}=\frac{m}{m+n}$, so we cannot have $m=n=1$. We will treat the case where $m\neq1$ does not divide $n$, the other case being treated similarly. Let $$\omega\coloneqq \genx^n\geny^{-1}\genx^{m-n}\geny\genx^n.$$ We shall show that $M_k(\omega)$ commutes with $M_k(\geny^2)$ only when $k=1$.

Let $E_k\coloneqq\{1,2,\dots,24\}$ be the set of edges of $\cD_k$ incident to $0$. The action of $\geny$ fixes the set $E_k$ on which it induces the cyclic permutation $(1,2,\dots,24)$, and every white vertex except $0$ has degree one so the action of $\geny$ is trivial on the complement of $E_k$.

We can write $E_k=E_k^{\odd}\sqcup E_k^{\even}$ as the disjoint union of the sets of respectively odd and even numbered edges incident to $0$, such that $\geny$ sends one to the other. The black vertices of $E_k^{\odd}$ are of degree $m$ except for the two black vertices of the edges $1$ and $13$ that are of degree $2m$. Therefore if $m$ does not divide some integer $l$ then $\genx^l$ sends every edge of $E_k^{\odd}$ to the complement of $E_k$, and otherwise the action of $\genx^m$ on $E_k^{\odd}$ corresponds to the sole transposition $(1,13)$. Similarly if $n$ does not divide $l$ then $\genx^l$ sends every edge of $E_k^{\even}$ to the complement of $E_k$, and the action of $\genx^n$ on $E_k^{\even}$ is the transposition $(2k,2k+12)$.

In particular, by hypothesis $n$ is not a multiple of $m$, so $m-n$ is not a multiple of $m$ either, hence both $\genx^n$ and $\genx^{m-n}$ send the edges of $E_k^{\odd}$ to the complement of $E_k$. However $\geny$ acts trivially on the latter, so $\genx^n\geny^{-1}\genx^{m-n}$ and $\genx^{m-n}\geny\genx^n$ both fix the set $E_k^{\odd}$ on which they induce the same action as $\genx^m$, i.e. the transposition $(1,13)$. Therefore the action of $\omega=\genx^n\geny^{-1}\genx^{m-n}\geny\genx^n$ is the same as that of $\genx^m\geny\genx^n$ on $E_k^{\odd}$ and the same as that of $\genx^n\geny^{-1}\genx^m$ on $E_k^{\even}$. See Figure \ref{ex1-trans}.

\begin{figure}[H]
  \centering
  \includegraphics[height=3cm]{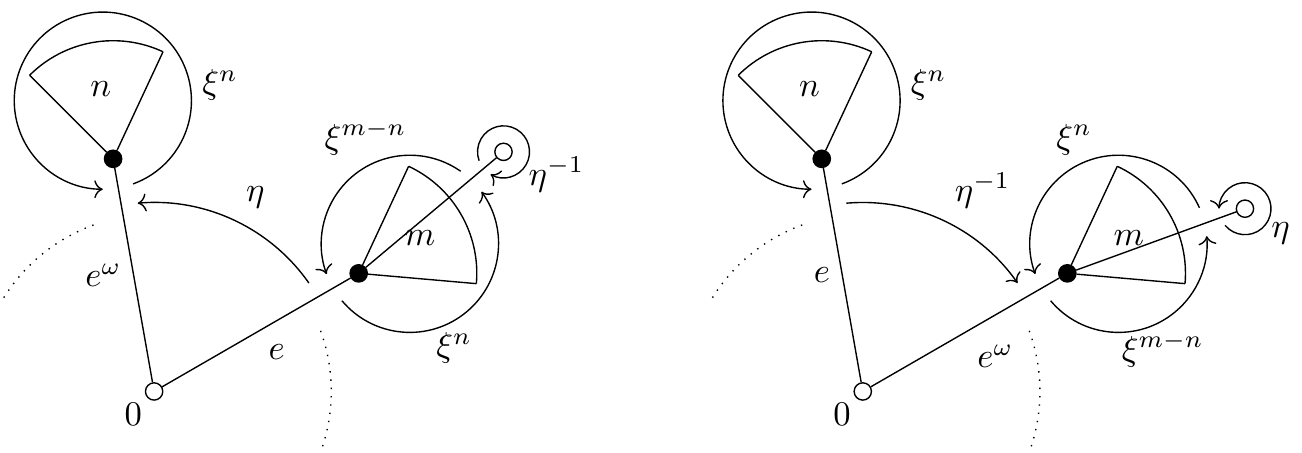}
  \caption{Action of $\omega$ on $E_k^{\odd}\setminus\{1,2k-1\}$ and on $E_k^{\even}\setminus\{2,2k\}$.}
  \label{ex1-trans}
\end{figure}

The action of $\omega$ fixes the set $E_k$ on which it induces the permutation $$M_k(\omega)|_{E_k}=(1,13)(2k,2k+12)\cdot(1,2)(3,4)\cdots(23,24)\cdot(1,13)(2k,2k+12)$$

Therefore for $k=1$,
\begin{align*}
  M_1(\omega)|_{E_1}&=(1,13)(2,14)\cdot(1,2)(3,4)\cdots(23,24)\cdot(1,13)(2,14)\\
                    &=(1,2)(3,4)\cdots(23,24)
\end{align*}
so $\omega$ and $\geny^2$ commute on $E_1$. Moreover $\geny$ acts trivially on the complement of $E_1$ so $M_1(\omega)|_{\cD_1\setminus E_1}$ and $M_1(\geny^2)|_{\cD_1\setminus E_1}$ automatically commute. Finally, we obtain that $M_1(\omega)$ and $M_1(\geny^2)$ commute.

\medskip
For $k=2$, we observe that $4^{\omega\geny^2}=15^{\geny^2}=17$ but $4^{\geny^2\omega}=6^\omega=5$. Similarly, for $3\leq k\leq6$, we observe that $1^{\omega\geny^2}=14^{\geny^2}=16$ but $1^{\geny^2\omega}=3^\omega=4$. We have thus shown that $M_k(\omega)$ and $M_k(\geny^2)$ commute only for $k=1$.

\medskip
This concludes the proof that $\widetilde{\cD}$ is a regular dessin with moduli field $\QQ(\zeta_3,\sqrt[3]{q})$.

\subsection{Regular dessins with moduli fields of the form $\QQ(\zeta_p,\sqrt[p]{q})$}
Let $p$ be an odd prime, and $q\in\QQ_{>0}$ a positive rational number that is not a $p$th power. In this example we will need an additional parameter $\gamma\in\QQ\setminus\{0\}$. Let
\begin{align*}
  C\colon y^2=x(x-(1-\zeta_p))(x-\gamma\sqrt[p]{q}).
\end{align*}

We construct the Belyi function $\beta\colon C\to\PP$ as the composition $\beta=\beta_2\circ\beta_1\circ\beta_0\circ\pi$ of the following maps.
\begin{enumerate}
\item $\pi\colon C\to\PP$ is the projection on the coordinate $x$, which ramifies over $\{0,1-\zeta_p,\gamma\sqrt[p]{q},\infty\}$.
\item $\beta_0\coloneqq X^{2p}\in\QQ[X]$, and $\Crit(\beta_0)=\{0,\infty\}$ so $\beta_0\circ\pi$ ramifies over $\{0,(1-\zeta_p)^{2p},\gamma^{2p}q^2,\infty\}$.
\item $\beta_1\in\QQ[X]$ is chosen independently of $\gamma$ such that $\Crit(\beta_1)\cup\{\beta_1((1-\zeta_p)^{2p})\}=\{0,1,\infty\}$, $\beta_1((1-\zeta_p)^{2p})=0<\beta_1(0)<1$ and $\beta_1'(0)>0$. The existence of $\beta_1$ verifying those conditions is assured by Proposition \ref{unramified-belyi-function} below. Under those assumptions $\beta_1\circ\beta_0\circ\pi$ ramifies over $\{0,1,\beta_1(0),\beta_1(\gamma^{2p}q^2),\infty\}$.
\item $\gamma\in\QQ_{>0}$ is then chosen small enough so that $\beta_1'>0$ on $[0,\gamma^{2p}q^2]$. This guarantees us that we have $0<\beta_1(0)<\beta_1(\gamma^{2p}q^2)<1$.
\item $\beta_2\coloneqq B_{r,s}\circ B_{m,n}$, where $(m,n)$ and $(r,s)$ are pairs of coprime positive integers such that $\beta_1(\gamma^{2p}q^2)=\frac{m}{m+n}$ and $B_{m,n}(\beta_1(0))=\frac{r}{r+s}$. Finally, $\beta=\beta_2\circ\beta_1\circ\beta_0\circ\pi$ ramifies over $\{0,1,\infty\}$.
\end{enumerate}

The pair $(C,\beta)$ is thus a Belyi pair, and we call $\cD$ the corresponding dessin. With the same arguments as before, the moduli field of $\cD$ is $\QQ(\zeta_p,\sqrt[p]{q})$, which is a nonabelian Galois extension of $\QQ$ with Galois group $$\Gal(\QQ(\zeta_p,\sqrt[p]{q})/\QQ)\cong\ZZn{p}\rtimes(\ZZn{p})^\times$$ generated by $\sigma\colon\zeta_p^i\sqrt[p]{q}\mapsto\zeta_p^{i+1}\sqrt[p]{q}$ and $\tau\colon\zeta_p^i\sqrt[p]{q}\mapsto\zeta_p^{gi}\sqrt[p]{q}$ where $g$ generates $(\ZZn{p})^\times$. We shall show that there exists $\gamma\in\QQ\setminus\{0\}$ such that the regular closure of the dessin $\cD$ thus obtained also has moduli field $\QQ(\zeta_p,\sqrt[p]{q})$.

\begin{remark}
  In the previous subsection we treated the case $p=3$. In that specific case we gave a simpler expression for $\beta$, mainly due to the fact that $\beta_0\circ\pi$ already had all of its critical values in $\QQ\cup\{\infty\}$. However in the general case we must use the intermediate map $\beta_1$ as well as the parameter $\gamma$ to conclude the proof.
\end{remark}

\bigskip
Let us first prove the existence of $\beta_1$.

\begin{proposition}\label{unramified-belyi-function}
  Let $E\subset\QQb\cap\RR\setminus\{0\}$ be a finite set. Then there exists $P\in\QQ[X]$ such that $P(E)\subseteq\{0\}$, $\Crit(P)\subseteq\{0,1\}$, $0<P(0)<1$ and $P'(0)>0$.
\end{proposition}
\begin{remark}
  In the context of this proposition we only deal with polynomials so for $P\in\QQ[X]$ we define $\Crit(P)\coloneqq\{P(z)|\;z\in\CC,P'(z)=0\}$, which does not include the point at infinity to simplify notations.
\end{remark}
\begin{proof}
To show this we will proceed similarly as in the proof of the {\em only if} part of Belyi's theorem, by applying additional transformations to ensure that $0<P(0)<1$. Let us first prove that we can reduce to the case where $E$ is a subset of rational numbers.

\begin{lemma}
  Let $E\subset\QQb\cap\RR\setminus\{0\}$ be a finite set fixed by $\AGG$. Then there exists $P\in\QQ[X]$ such that $P(0)=0$ and $\Crit(P)\cup P(E)\subset\QQ\setminus\{0\}$.
\end{lemma}
\begin{proof}
  Let $\{a_1,\dots,a_m\}=E\cap\QQ$ and $\{b_1,\dots,b_n\}=E\setminus\QQ$. We construct $P$ by induction on the number $n$ of non rational elements of $E$.

  For $\alpha\in\QQ$, define $F_\alpha,G_\alpha\in\QQ[X]$ by $$F_\alpha\coloneqq\prod_{j=1}^n(X-(b_j-\alpha)^2)\quad\text{and}\quad G_\alpha\coloneqq F_\alpha((X-\alpha)^2)=\prod_{j=1}^n(X-b_j)(X+b_j-2\alpha).$$

  Let us first assume that there exists $\alpha\in\QQ$ such that $G_\alpha(0)\not\in\Crit(G_\alpha)\cup G_\alpha(E)$. Define $P_1(X)\coloneqq G_\alpha(X)-G_\alpha(0)\in\QQ[X]$, then $P_1(0)=0\not\in E'\coloneqq\Crit(P_1)\cup P_1(E)\subset\QQb\cap\RR\setminus\{0\}$. Note that $E'$ is stable under the action of $\AGG$, and $|E'\setminus\QQ|=|\Crit(F_\alpha)\cup F_\alpha(0)\setminus\QQ|=|\Crit(F_\alpha)\setminus\QQ|<\deg F_\alpha=n$. By induction, there exists $P_2\in\QQ[X]$ such that $P_2(0)=0$ and $\Crit(P_2)\cup P_2(E')\subset\QQ\setminus\{0\}$. Now $P\coloneqq P_2\circ P_1$ has the desired properties, since $P(0)=0$ and $\Crit(P)\cup P(E)=\Crit(P_2)\cup P_2(\Crit(P_1))\cup P_2(P_1(E))=\Crit(P_2)\cup P_2(E')\subset\QQ\setminus\{0\}$.

  \medskip
  Let us now prove that there exists $\alpha\in\QQ$ such that $G_\alpha(0)\not\in\Crit(G_\alpha)\cup G_\alpha(E)$. Let us first treat the case where $0<b_1<b_2,\dots,b_n$. When $\alpha$ approaches $\frac{1}{2}$ , $G_\alpha(0)=\prod_{j=1}^n-b_j(b_j-2\alpha)$ approaches $0$ but the critical values of $G_\alpha$ do not. Indeed, $\Crit(G_\alpha)=\Crit F_\alpha\cup F_\alpha(\Crit((X-\alpha)^2))=\Crit(F_\alpha)\cup\{F_\alpha(0)\}$; $F_\alpha(0)$ approaches $F_{\frac{b_1}{2}}(0)\neq0$, and since $F_{\frac{b_1}{2}}$ does not have multiple roots, the critical values of $F_\alpha$ approach the critical values of $F_{\frac{b_1}{2}}$ which are all non zero. Therefore for $\alpha\neq\frac{b_1}{2}$ in the neighborhood of $\frac{b_1}{2}$ we have $G_\alpha(0)\not\in\Crit(G_\alpha)$. Moreover $G_\alpha(0),G_\alpha(a_1),\dots,G_\alpha(a_m)$ are all distinct polynomials in the indeterminate $\alpha$, so they coincide at only finitely many points. In particular for $\alpha\neq\frac{b_1}{2}$ in the neighborhood of $\frac{b_1}{2}$ we have $G_\alpha(0)\not\in\{G_\alpha(a_1),\dots,G_\alpha(a_m)\}$. Since $\alpha\in\QQ$ we also have $G_\alpha(0)\neq0=G_\alpha(b_1)=\dots=G_\alpha(b_n)$ hence $G_\alpha(0)\not\in G_\alpha(E)$, proving the existence of $\alpha$ as desired.

  \medskip
  Let us now treat the general case where $b_1,\dots,b_n$ are not assumed to be positive by reducing it to the previous case. For ${\alpha'}\in\QQ$, define $H_{\alpha'}\in\QQ[X]$ by $$H_{\alpha'}\coloneqq (X-{\alpha'})^2-{\alpha'}^2\in\QQ[X].$$ Note that $\Crit(H_{\alpha'})=\{-{\alpha'}\}$. For ${\alpha'}>0$ sufficiently small we have $-{\alpha'}^2<H_{\alpha'}(0)=0<H_{\alpha'}(a_1),\dots,H_{\alpha'}(a_m),H_{\alpha'}(b_1),\dots,H_{\alpha'}(b_n)$. Let $E''\coloneqq\Crit(H_{\alpha'})\cup H_{\alpha'}(E)$. The set $E''$ is a finite subset of $\QQb\cap\RR\setminus\{0\}$ fixed by $\AGG$, and $E''$ has at most $n$ non rational elements, which are all positive. By the above, there exists $P_3\in\QQ[X]$ such that $\Crit(P_3)\cup P_3(E'')\subset\QQ\setminus\{0\}$ and $P_3(0)=0$. Then $P\coloneqq P_3\circ H_{\alpha'}$ has the desired properties, since $P(0)=0$ and $\Crit(P)\cup P(E)=\Crit(P_3)\cup P_3(\Crit(H_{\alpha'}))\cup P_3(H_{\alpha'}(E))=\Crit(P_3)\cup P_3(E'')\subset\QQ\setminus\{0\}$.
\end{proof}

Let us denote by $P_1$ the polynomial obtained using this lemma, which verifies $P_1(0)=0$ and $E'\coloneqq\Crit(P_1)\cup P_1(E)\subset\QQ\setminus\{0\}$. We can further assume that $P_1'(0)>0$ by taking $(-P_1)$ if necessary. We now send the points $E'$ to $\{0,1\}$.

\begin{lemma}
Let $E\subset\QQ\setminus\{0\}$ a finite set. Then there exists $P\in\QQ[X]$ such that $P(E)\subseteq\{0\}$, $\Crit(P)\subseteq\{0,1\}$, $0<P(0)<1$ and $P'(0)>0$.
\end{lemma}
\begin{proof}
  For $\alpha\in\QQ$, let $F_\alpha\coloneqq(X-\alpha)^2\in\QQ[X]$, and note that $\Crit(F_\alpha)=\{0\}$. There exists $\alpha<0$ sufficiently small such that $0<F_\alpha(0)<F_\alpha(a)$ for all $a\in E$. We take $$F\coloneqq\frac{F_\alpha}{\max_{a\in E}F_\alpha(a)}.$$ Let $\{a_1,\dots,a_l\}=F(E)$ such that $0<F(0)<a_1<\dots<a_l=1$. We also add a rational point $a_0\in\QQ$ such that $F(0)<a_0<a_1$.

  \medskip
  Let $m$ and $n$ be the coprime positive integers such that $a_{l-1}=\frac{m}{m+n}$. We recall that $B_{m,n}$ verifies $\Crit(B_{m,n})=\{0,1\}$, $B_{m,n}(0)=B_{m,n}(1)=0$, $B_{m,n}(\frac{m}{m+n})=1$, and $B_{m,n}$ is strictly increasing between $0$ and $\frac{m}{m+n}$. Let $P_1\coloneqq B_{m,n}$, then $\Crit(P_1)=\{0,1\}$ and $0<P_1\circ F(0)<P_1(a_0)<\dots<P_1(a_{l-1})=1$. There is one point fewer than before, so we can iteratively construct $P_2,\dots,P_l$ in the same way, so that $P\coloneqq P_l\circ\cdots\circ P_1$ verifies $\Crit(P)\subseteq\{0,1\}$, $P(a_1)=\cdots=P(a_l)=0<P(F(0))<1=P(a_0)$ and $P'(F(0))>0$. Therefore $P\circ F$ has the desired properties.
\end{proof}

Let us denote by $P_2$ the polynomial obtained using this lemma with the finite set $E'$ obtained previously. Then the polynomial $P\coloneqq P_2\circ P_1$ verifies $P(E)\subseteq\{0\}$, $\Crit(P)\subseteq\{0,1\}$, $0<P(0)<1$ and $P'(0)>0$, thus concluding the proof of Proposition \ref{unramified-belyi-function}.
\end{proof}

We can now use Proposition \ref{unramified-belyi-function} with the finite set $$E\coloneqq\{(1-\zeta_p^k)^{2p}\}_{1\leq k\leq\frac{p-1}{2}}$$ to obtain the map $\beta_1$ as desired. For $1\leq k\leq\frac{p-1}{2}$ we have $(1-\zeta_p^k)^{2p}=(|1-\zeta_p^k|\zeta_{2p}^{2k-1})^{2p}=|1-\zeta_p^k|^{2p}\in\RR$ so $E\subset\QQb\cap\RR$, and the set $E$ is the Galois orbit of $(1-\zeta_p)^{2p}$ so it is fixed by $\AGG$, hence $E$ verifies the conditions of Proposition \ref{unramified-belyi-function}.

Let us denote by $\cD(\beta_1)$ the dessin corresponding to the Belyi pair $(\PP,\beta_1)$. The Belyi pair $(\PP,\beta_1)$ is fixed by the action of the complex conjugation, so the embedding of $\cD(\beta_1)$ on $\PP$ admits a symmetry along the real line. Moreover the Belyi function $\beta_1$ is a polynomial, so $\cD(\beta_1)\cap\RR$ is a (graph theoretic) path. Let $v_l<\dots<v_1$ be the negative vertices on the path, and let $e_k$ denote the edge $(v_{k-1},v_k)$. By hypothesis $\beta_1'(0)>0$ so $v_1$ is a black vertex, and for $k<l$, the vertex $v_k$ is of even degree $2d_k$. We then have $e_k^{\genx^{d_k}}=e_{k+1}$ and $e_{k+1}^{\genx^{d_k}}=e_k$ if $k$ is odd, or $e_k^{\geny^{d_k}}=e_{k+1}$ and $e_{k+1}^{\geny^{d_k}}=e_k$ if $k$ is even. See Figure \ref{construction-beta-ex3}.

\medskip
As remarked earlier, the Galois orbit of $(1-\zeta_p)^{2p}$ is $\{(1-\zeta_p^k)^{2p}\}_{1\leq k\leq\frac{p-1}{2}}\subset\RR_-$, and $(1-\zeta_p^{\frac{p-1}{2}})^{2p}<\cdots<(1-\zeta_p)^{2p}<0$. By construction $\beta_1((1-\zeta_p)^{2p})=0$, so $(1-\zeta_p)^{2p}$ and all its Galois conjugates are black vertices of $\cD(\beta_1)$ lying on the path $(v_1,\cdots,v_l)$. Let $t>0$ be the index such that $v_t=(1-\zeta_p)^{2p}$, and $v_t$ is a black vertex so $t$ is odd. Then $$\mu_0\coloneqq\genx^{d_1}\geny^{d_2}\cdots\geny^{d_{t-1}}\genx^{2d_{t}}\geny^{d_{t-1}}\cdots\geny^{d_2}\genx^{d_1}$$ fixes the edge $e_1$ (Figure \ref{construction-beta-ex3}).

\begin{figure}[H]
  \centering
  \includegraphics[height=2.2cm]{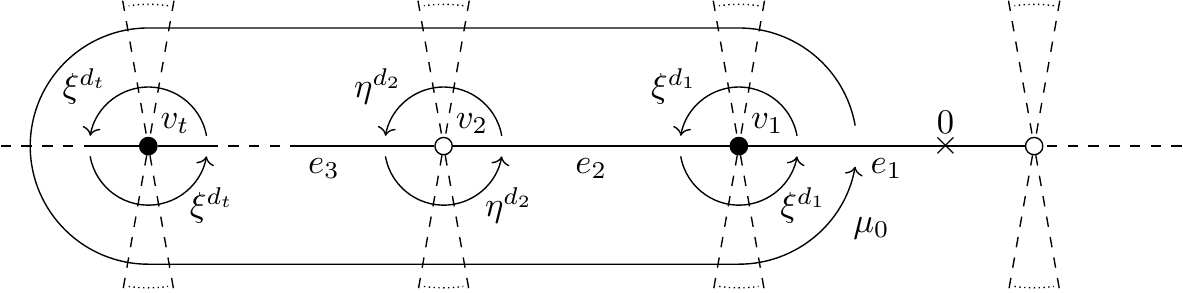}
  \caption{Dessin $\cD(\beta_1)$ corresponding to $(\PP,\beta_1)$.}
  \label{construction-beta-ex3}
\end{figure}

Let $\gamma>0$ small enough so that $\beta_1'>0$ on $[0,\gamma^{2p}q^2]$. Let us next draw the dessin $\cD(\beta_2)$ corresponding to the Belyi pair $(\PP,\beta_2=B_{r,s}\circ B_{m,n})$. See Figure \ref{ex2-beta2}.

\begin{figure}[H]
  \centering
  \includegraphics[height=4.2cm]{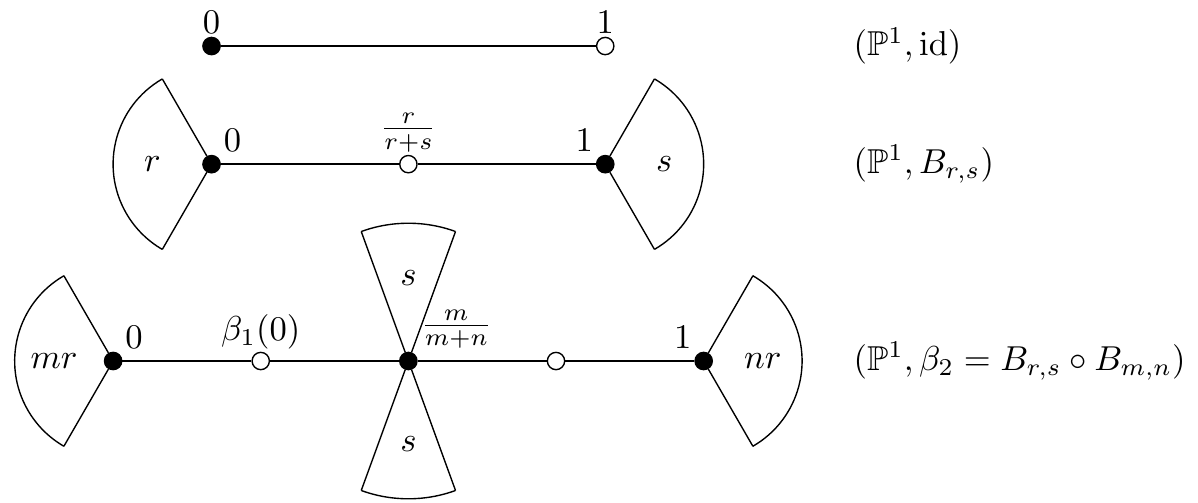}
  \caption{Dessin $\cD(\beta_2)$ corresponding to $(\PP,\beta_2)$.}
  \label{ex2-beta2}
\end{figure}

By lifting the dessin $\cD(\beta_2)$ along $\beta_1$ we obtain the dessin $\cD(\beta_2\circ\beta_1)$ corresponding to the Belyi pair $(\PP,\beta_2\circ\beta_1)$. This amounts to replacing each edge of $\cD(\beta_1)$ by a copy of $\cD(\beta_2)$. Note that the degrees of the black and white vertices are thus multiplied by $mr$ and $nr$, respectively. Analogously to $\mu_0$ we define
\begin{gather*}
\begin{split}
  \mu\coloneqq\MoveEqLeft(\genx^{mrd_1}\geny\genx^s\geny)(\genx^{nrd_2}\geny\genx^s\geny)\cdots(\genx^{mrd_{t-2}}\geny\genx^s\geny)(\genx^{nrd_{t-1}}\geny\genx^s\geny)\\
  &\cdot(\genx^{2mrd_t}\geny\genx^s\geny)(\genx^{nrd_{t-1}}\geny\genx^s\geny)(\genx^{mrd_{t-2}}\geny\genx^s\geny)\cdots(\genx^{nrd_2}\geny\genx^s\geny)\genx^{mrd_1}
\end{split}
\end{gather*}
  and we verify again that $\mu$ fixes the edge $(0,v_1)$. Note also that $\genx^{2s}$ fixes the edge $(0,\gamma\sqrt[p]{q})$. See Figure \ref{ex2-beta21}.

\begin{figure}[H]
  \centering
  \includegraphics[height=3cm]{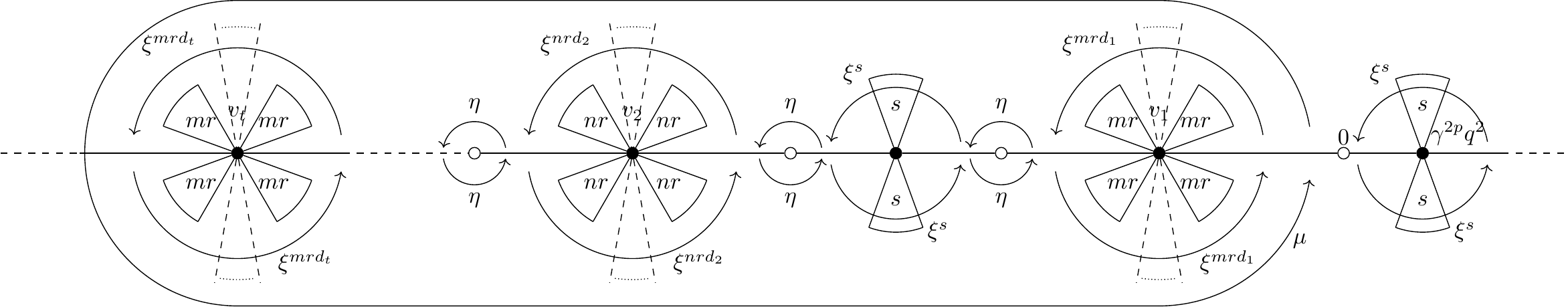}
  \caption{Dessin $\cD(\beta_2\circ\beta_1)$ corresponding to $(\PP,\beta_2\circ\beta_1)$.}
  \label{ex2-beta21}
\end{figure}

Let $\cD_0$ be the dessin corresponding to the Belyi pair $(\PP,\beta_2\circ\beta_1\circ\beta_0)$. To simplify the representations of the dessins we only show the vertices $0$, $\zeta_{2p}^k(1-\zeta_p)$, $1-\zeta_p^k$, and $\zeta_{2p}^k\gamma\sqrt[p]{q}$. We decorate the vertices $\zeta_{2p}^k(1-\zeta_p)$ (which map to $(1-\zeta_p)^{2p}\in\RR_-$ by $\beta_0$) and $\zeta_{2p}^k\gamma\sqrt[p]{q}$ (which map to $\gamma^{2p}q^2\in\RR_+$ by $\beta_0$) respectively with the symbols $\treem$ and $\treep$ to distinguish them. See Figure \ref{ex2-D0}.

\ignore{The dessin $\cD$ corresponding to $(C,\beta)$ will then be obtained by lifting $\cD_0$ to $C$.}

\begin{figure}[H]
  \centering
  \includegraphics[height=8.5cm]{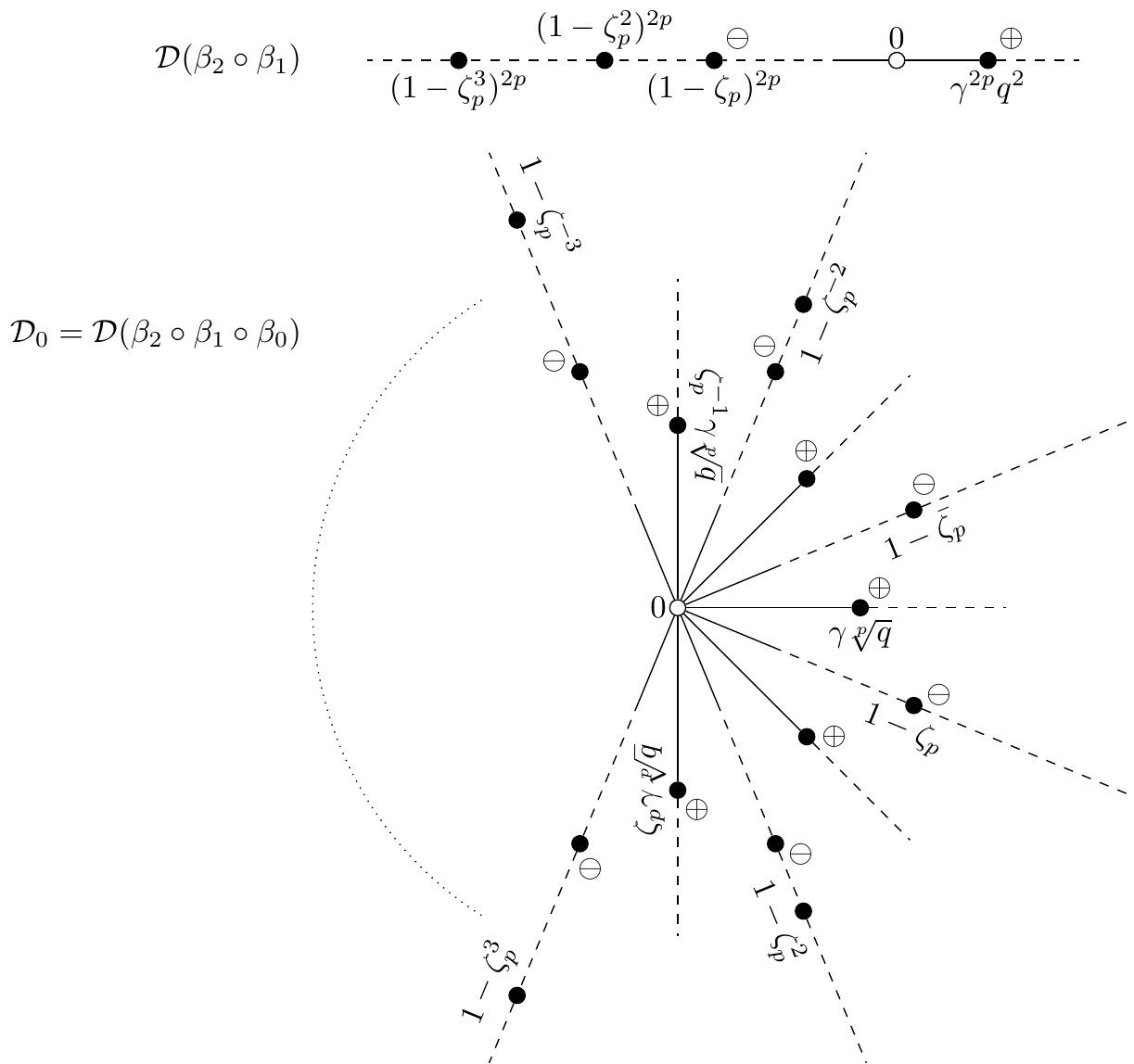}
  \caption{Dessin $\cD_0$ corresponding to $(\PP,\beta_2\circ\beta_1\circ\beta_0)$.}
  \label{ex2-D0}
\end{figure}

We may now draw the Galois conjugates $\cD^\sigma$ for $\sigma\in\AGG$ by lifting the dessin $\cD_0$ along the projection $\pi$, by treating separately the cases $\sigma(\zeta_p)\in\{\zeta_p,\bar{\zeta_p}\}$ and $\sigma(\zeta_p)\in\{\zeta_p^2,\dots,\zeta_p^{p-2}\}$. We call the dessins respectively $\cD_k$ and $\cD_k^j$, see Figure \ref{ex2-D}. We identify the outermost edges on opposite sides in the representations.

\begin{figure}[H]
  \begin{subfigure}{0.5\hsize}
    \centering
    \includegraphics[width=\hsize]{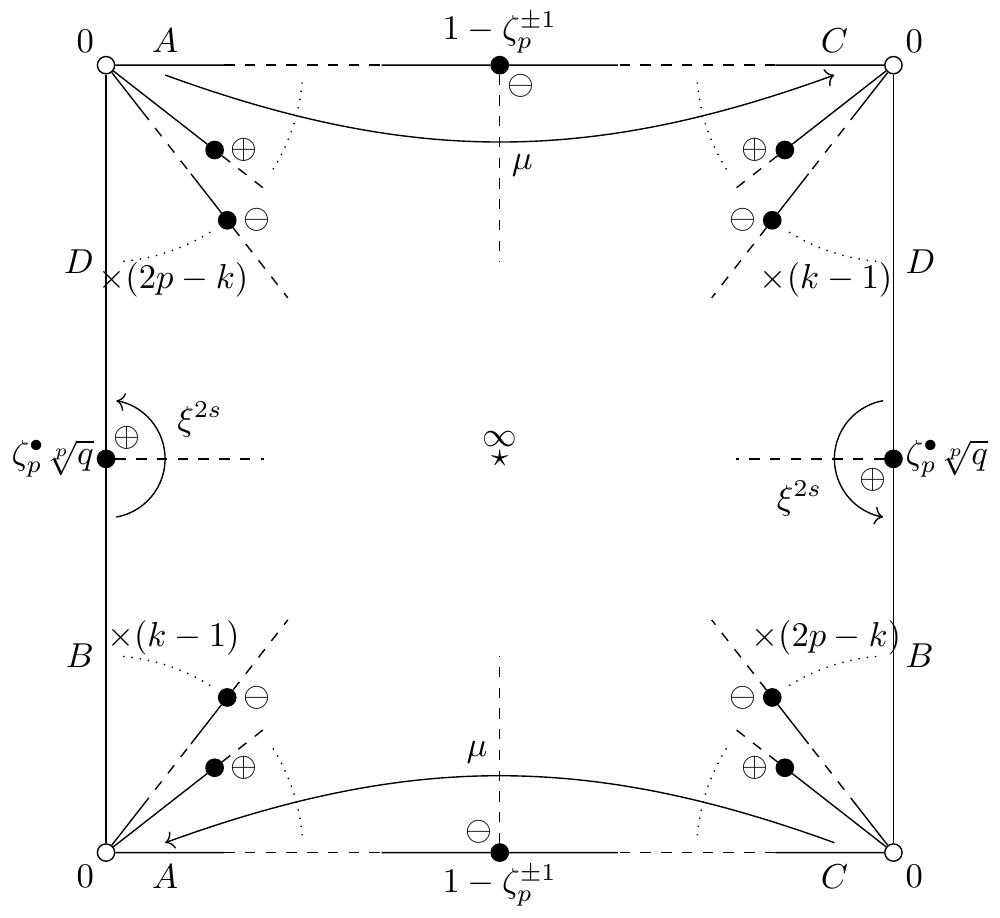}
    \caption{$\cD_k\coloneqq\cD^\sigma$ with $\sigma(\zeta_p)=\zeta_p^{\pm1}$.}
    \label{ex2-Dk}
  \end{subfigure}
  \begin{subfigure}{0.5\hsize}
    \centering
    \includegraphics[width=\hsize]{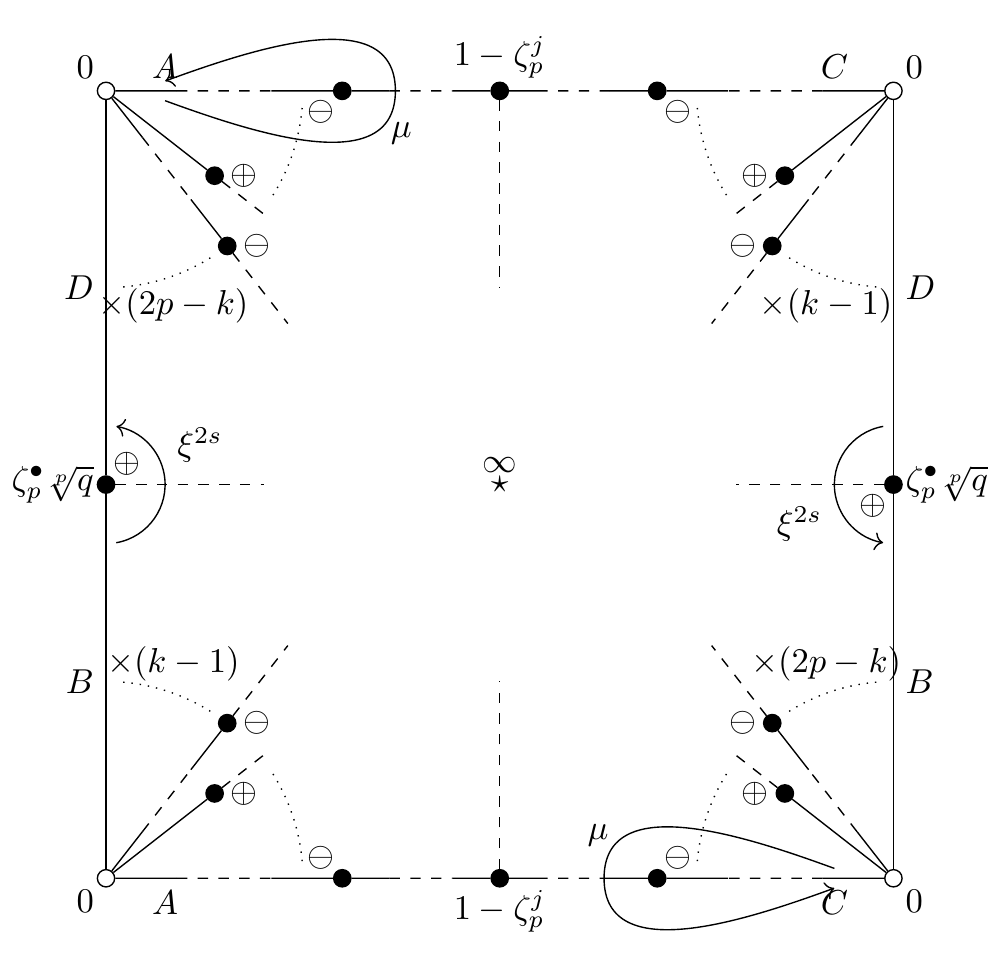}
    \caption{$\cD_k^j\coloneqq\cD^\sigma$ with $\sigma(\zeta_p)=\zeta_p^j$ where $j\in\{2,\dots,p-2\}$.}
    \label{ex2-Dkj}
  \end{subfigure}
  \caption{Dessins (a) $\cD_k$ and (b) $\cD_k^j$ in the Galois orbit of $\cD$.}
  \label{ex2-D}
\end{figure}

For all $k$, we have in fact $\cD_{2k-1}=\cD^\sigma$ where $\sigma\colon(\zeta_p,\sqrt[p]{q})\mapsto(\zeta_p,\zeta_p^{k-1}\sqrt[p]{q})$, and $\cD_{2k}=\cD^\sigma$ where $\sigma\colon(\zeta_p,\sqrt[p]{q})\mapsto(\bar{\zeta_p},\zeta_p^{k}\sqrt[p]{q})$. We have similar expressions for the dessins $\cD_k^j$.

\medskip
\ignore{Let $k$ be fixed, and let us consider one of $\cD_k$ or $\cD_k^j$. Let $E$ denote the set of edges incident to $0$. The action of $\geny$ induces a cyclic permutation of the edges of $E$.}
Let $k$ be fixed, and let us consider the dessin $\cD_k$. Let $A$ denote one of the two edges incident to $0$ and on the path to the ramification point $1-\zeta_p^{\pm1}$. We also call $B\coloneqq A^{\geny^{2k-1}},C\coloneqq A^{4p},D\coloneqq C^{\geny^{2k-1}}$ (See Figure \ref{ex2-Dk}). Let $E$ denote the set of edges incident to $0$. The action of $\geny$ induces the cyclic permutation of the edges of $E=\{A^{\geny^i}\}_{0\leq i<8p}$. Furthermore by construction every white vertex aside from $0$ has degree $1$ or $2$, so $\geny^2$ fixes every edge in the complement of $E$. We can write $E=E^{\treem}\sqcup E^{\treep}$ as the disjoint union of $E^{\treem}\coloneqq\{A^{2i}\}_{0\leq i<4p}$ and $E^{\treep}\coloneqq\{B^{2i}\}_{0\leq i<4p}$, such that $\geny$ sends one to the other. The action of $\mu$ on $E^\treem$ is the transposition $(A,C)$, and similarly the action of $\genx^{2s}$ on $E^\treep$ is the transposition $(B,D)$.

We do the same for the dessins of the form $\cD_k^j$, with the only difference that this time the action of $\mu$ on $E^\treem$ is trivial, including on the edges $A$ and $C$ (see Figure \ref{ex2-Dkj}).

\bigskip
We are almost in the same configuration as in the first example. We define analogously $$\omega\coloneqq \mu\geny \mu^{-1}\genx^{2s}\geny^{-1}\mu,$$ and we shall prove that for some choices of $\gamma$, the actions of $\omega$ and of $\geny^2$ commute only for $\cD_1$. To reproduce the proof in the first example we need only show that for some choice of $\gamma$ the actions of $\mu\geny\mu^{-1}\genx^{2s}$ and $\mu^{-1}\genx^{2s}\geny^{-1}\mu$ on the set $E^\treep$ is the same as that of $\genx^{2s}$.

Note that for any edge $e\in E^\treep$, the edge $e^{\genx^i}$ is fixed by $\geny$ if $i$ is not a multiple of $s$. To that end we shall show that for some choice of $\gamma$ the action of $\mu$ on $E^\treep$ is the same as that of $\genx^\delta$, where $\delta$ is the number of occurences of $\genx$ in the word $\mu$, and then that $\delta$ is not a multiple of $s$.

\medskip
We define the words $\rho_1,\rho'_1,\rho_2,\rho'_2,\dots,\rho_{2t-2},\rho'_{2t-2}\in F_2$ to be the increasing subsequence of the prefixes ending in $\eta$ of the word $\mu$ defined above, such that $\rho_1\coloneqq\genx^{mrd_1}\geny$, $\rho_1'\coloneqq\rho_1\genx^s\geny$, $\rho_2\coloneqq\rho_1'\genx^{nrd_2}\geny$, $\rho_2'\coloneqq\rho_2\genx^s\geny$, etc., and $\mu=\rho_{2t-2}'\genx^{mrd_1}$. We shall show by induction that for some choice of $\gamma$ the action of $\rho_i$ (resp. $\rho_i'$) is the same as the action of $\genx^{\delta_i}$ (resp. $\genx^{\delta_i'}$), where $\delta_i$ (resp. $\delta_i'$) is the number of occurences of $\genx$ in the word $\rho_i$ (resp. $\rho_i'$). By induction it suffices to show that $\delta_i$, $\delta_i'$ are not multiples of $s$. Modulo $s$ we have $\delta_i\equiv\delta_i'$ equal to the non empty partial sum of $$mrd_1+nrd_2+\cdots+mrd_{t-2}+nrd_{t-1}+2mrd_t+nrd_{t-1}+mrd_{t-2}+\cdots+nrd_2+mrd_1$$ consisting of the first $i$ terms.

\bigskip
To proceed we shall use the following result, but let us first introduce some notations. Let $P=\sum_{i=0}^dc_iX^i\in\ZZ[X]$ and $c\in\ZZ_{>0}$ such that $\beta_1=\frac{P}{c}$. Note that $P$ and $c$ do not depend on the choice of $\gamma$, and $0<\beta_1(0)=\frac{P(0)}{c}<1$ so $0<c_0,c-c_0$. We define $$\alpha\coloneqq v_2(\gamma^{2p}q^2),\quad\nu\coloneqq v_2(c_0)+v_2(c-c_0),$$ where $v_2$ denotes the $2$-valuation.

\begin{lemma}
  If $\alpha>\nu$, then there exists $e\in\ZZ$ such that $em\equiv c_0\mod2^\alpha$ and $en\equiv c-c_0\mod2^\alpha$, and $v_2(s)\geq\alpha-\nu$.
 \end{lemma}
 \begin{proof}
   Let $a,b\in\ZZ$ coprime such that $\gamma^{2p}q^2=\frac{a}{b}2^\alpha$.

   Firstly, $$\frac{m}{m+n}=\beta_1(\frac{a}{b}2^\alpha)=\frac{P(\frac{a}{b}2^\alpha)}{c}=\frac{\sum_{i=0}^dc_ia^i2^{\alpha i}b^{d-i}}{b^dc},$$ so there exists $f\in\ZZ$ such that $fm=\sum_{i=0}^dc_ia^i2^{\alpha i}b^{d-i}$ and $f(m+n)=b^dc$, so $$em\equiv c_0\mod2^\alpha,\quad en\equiv c-c_0\mod2^\alpha$$ for $e\in\ZZ$ such that $eb^d\equiv f\mod2^\alpha$.

   Secondly, $$\frac{r}{r+s}=B_{m,n}(\beta_1(0))=\frac{\beta_1(0)^m(1-\beta_1(0))^n}{\beta_1(\frac{a}{b}2^\alpha)^m(1-\beta_1(\frac{a}{b}2^\alpha))^n}=\frac{b^{d(m+n)}c_0^m(c-c_0)^n}{(b^dP(\frac{a}{b}2^\alpha))^m(b^dc-b^dP(\frac{a}{b}2^\alpha))^n},$$ so there exists $g\in\ZZ$ such that $gr=b^{d(m+n)}c_0^m(c-c_0)^n$ and $g(r+s)=(b^dP(\frac{a}{b}2^\alpha))^m(b^dc-b^dP(\frac{a}{b}2^\alpha))^n$. In the expansion of $(b^dP(\frac{a}{b}2^\alpha))^m$, aside from the constant term $b^{dm}c_0^m$, every other term is a multiple of an integer of the form $c_0^i2^{\alpha j}$ with $i\leq m-1$ and $j\geq m-i$. By hypothesis $\alpha>\nu\geq v_2(c_0)$, so those other terms are all multiples of $2^{\alpha+(m-1)v_2(c_0)}$, hence there exists $A\in\ZZ$ such that $(b^dP(\frac{a}{b}2^\alpha))^m=b^{dm}c_0^m+A2^{\alpha+(m-1)v_2(c_0)}$. Similarly there exists $B\in\ZZ$ such that $(b^dc-b^dP(\frac{a}{b}2^\alpha))^n=b^{dn}(c-c_0)^n+B2^{\alpha+(n-1)v_2(c-c_0)}$. Then $g(r+s)=b^{d(m+n)}c_0^m(c-c_0)^n+C2^{\alpha+(m-1)v_2(c_0)+(n-1)v_2(c-c_0)}$ for some $C\in\ZZ$, so $gr=b^{d(m+n)}c_0^m(c-c_0)^n$ and $gs=C2^{\alpha+(m-1)v_2(c_0)+(n-1)v_2(c-c_0)}$. The integers $r$ and $s$ are coprime, so after dividing $gr$ and $gs$ by their greatest common dividor we obtain that $$v_2(s)\geq \alpha-\nu>0.$$
 \end{proof}

 Using this lemma, we know that if $\alpha>\nu$, then there exists $e\in\ZZ$ such that $em\equiv c_0\mod2^\alpha$ and $en\equiv c-c_0\mod2^\alpha$, $v_2(s)\geq\alpha-\nu$ where $\nu$ does not depend on $\gamma$, and $r$ is coprime to $s$ so is not a multiple of $2$. Therefore there exists $e'\in\ZZ$ such that $e'mr\equiv c_0\mod2^\alpha$ and $e'nr\equiv c-c_0\mod2^\alpha$. Moreover $2^{\alpha-\nu}$ is a common divisor of $2^\alpha$ and $s$, so by the above modulo $2^{\alpha-\nu}$ we have $e'\delta_i\equiv e'\delta_i'$ equal to the non empty partial sum $\widetilde{\delta_i}$ consisting of the first $i$ terms of the sum $$c_0d_1+(c-c_0)d_2+\cdots+c_0d_{t-2}+(c-c_0)d_{t-1}+2c_0d_t+(c-c_0)d_{t-1}+c_0d_{t-2}+\cdots+(c-c_0)d_2+c_0d_1.$$ Similarly $e'\delta$ is equal modulo $2^{\alpha-\nu}$ to the whole sum $$\widetilde{\delta}\coloneqq2(c_0d_1+(c-c_0)d_2+\cdots+c_0d_{t-2}+(c-c_0)d_{t-1}+c_0d_t).$$

 By construction $c_0,c-c_0,d_i$ are positive and do not depend on the choice of $\gamma$, so $0<c_0d_1\leq\widetilde{\delta}_i\leq\widetilde{\delta}$, thus for any choice of $\gamma$ such that $\alpha>\nu$ and $\widetilde{\delta}<2^{\alpha-\nu}$ (for instance $\gamma=\frac{2^u}{2^v+1}$ with $1\ll u\ll v$), we obtain $\widetilde{\delta}_i,\widetilde{\delta}\not\equiv 0\mod 2^{\alpha-\nu}$, and in consequence $\delta_i,\delta_i'$ and $\delta$ are not multiples of $s$. Therefore we can now conclude by induction that the actions of $\rho_i$ and $\rho_i'$ are the same as that of $\genx^{\delta_i}$ and $\genx^{\delta_i'}$, respectively. Indeed, $\delta_1$ is not a multiple of $s$ so $\rho_1=\genx^{\delta_1}\geny$ and $\genx^{\delta_1}$ have the same action on $E^\treep$. If $\rho_i$ has the same action as $\genx^{\delta_i}$ on $E^\treep$, then $\rho_i'=\rho_i\genx^s\geny$ has the same action as $\genx^{\delta_i}\genx^s\geny=\genx^{\delta_i'}\geny$ on $E^\treep$, and also the same action as $\genx^{\delta_i'}$ because $\delta_i'$ is not a multiple of $s$. Similarly, if $\rho_i'$ has the same action as $\genx^{\delta_i'}$ on $E^\treep$, then $\rho_{i+1}$ has the same action as $\genx^{\delta_{i+1}}\geny$ on $E^\treep$, and also the same action as $\genx^{\delta_{i+1}}$ because $\delta_{i+1}$ is not a multiple of $s$.

\medskip
We have thus proved that $\mu$ has the same action as $\genx^{\delta}$ on $E^\treep$, and by symmetry $\mu^{-1}$ has the same action as $\genx^{-\delta}$ on $E^\treep$. And $\delta$ and $2s-\delta$ are not multiples of $s$, so $\mu\geny\mu^{-1}\genx^{2s}$ and $\mu^{-1}\genx^{2s}\geny^{-1}\mu$ have the same action as $\genx^{2s}$ on $E^\treep$, as announced. We shall now observe the action of $\omega=\mu\geny\mu^{-1}\genx^{2s}\geny^{-1}\mu$ on $E$. Let $M_k$ and $M_k^j$ denote the monodromy maps of the dessins $\cD_k$ and $\cD_k^j$.

\medskip
For the dessins $\cD_k$ for $1\leq k\leq 2p$, the action of $\mu$ on $E^\treem$ is the transposition $(A,C)$, and the action of $\genx^{2s}$ on $E^\treep$ is the transposition $(B,D)$, therefore the action of $\omega$ fixes the set $E$ on which it induces the permutation $$M_k(\omega)|_E=(A,C)(B,D)\cdot\prod_{i=0}^{4p-1}(A^{\geny^{2i}},A^{\geny^{2i+1}})\cdot(A,C)(B,D).$$
Hence for $k=1$,
\begin{align*}
  M_1(\omega)|_E&=(A,A^{\geny^{4p}})(A^\geny,A^{\geny^{4p+1}})\cdot\prod_{i=0}^{4p-1}(A^{\geny^{2i}},A^{\geny^{2i+1}})\cdot(A,A^{\geny^{4p}})(A^\geny,A^{\geny^{4p+1}})\\
  &=\prod_{i=0}^{4p-1}(A^{\geny^{2i}},A^{\geny^{2i+1}})
\end{align*}
so $\omega$ and $\geny^2$ commute on $E$. Moreover $\geny^2$ acts trivially on the complement of $E$, so finally $M_1(\omega)$ and $M_1(\geny^2)$ commute.

\medskip
For $k=2$, we observe that $B^{\omega\geny^2}=D^{\geny^{-1}\geny^2}=D^{\geny}$ but $B^{\geny^2\omega}=B^{\geny^2\geny^{-1}}=B^{\geny}$. Similarly, for $3\leq k\leq 2p$, we observe that $A^{\omega\geny^2}=C^{\geny\geny^2}=C^{\geny^3}$ but $A^{\geny^2\omega}=A^{\geny^2\geny}=A^{\geny^3}$. Therefore $M_k(\omega)$ and $M_k(\geny^2)$ do not commute for $2\leq k\leq 2p$.

\medskip
For the dessins $\cD_k^j$ for $1\leq k\leq 2p$ and $2\leq j\leq \frac{p-1}{2}$, $\genx^{2s}$ on $E^\treep$ is the transposition $(B,D)$, and $\mu$ acts trivially on $E^\treem$, therefore the action of $\omega$ fixes the set $E$ on which it induces the permutation $$M_k^j(\omega)|_E=(B,D)\cdot\prod_{i=0}^{4p-1}(A^{\geny^{2i}},A^{\geny^{2i+1}})\cdot(B,D).$$
Hence we observe that $B^{\omega\geny^2}=D^{\geny^{-1}\geny^2}=D^{\geny}$ but $B^{\geny^2\omega}=B^{\geny^2\geny^{-1}}=B^{\geny}$, so $M_k^j(\omega)$ and $M_k^j(\geny^2)$ do not commute.

\medskip
We have thus shown that the actions of $\omega$ and $\geny^2$ commute only for $\cD_1$, this concludes the proof that $\widetilde{\cD}$ is a regular dessin with moduli field $\QQ(\zeta_p,\sqrt[p]{q})$.

\subsection{Regular dessin with moduli field $\QQ(\zeta_3,\sqrt[3]{3})$}\label{ex3}
Finally, let us exhibit a regular dessin with moduli field $\QQ(\zeta_3,\sqrt[3]{3})$ of smaller degree by choosing a Belyi map that is a rational function instead of a polynomial as was done in the previous subsections. Let
\begin{align*}
  C&\colon y^2=x(x-(1-\zeta_3))(x-\sqrt[3]{3}),\\
  \beta&\colon C\to\PP,(x,y)\mapsto\frac{(x+3^3)^3}{3^5(x-3^2)^2}.
\end{align*}

The function $\beta$ is given by the composition of the following maps $\beta=\beta_1\circ\beta_0\circ\pi$.
\begin{enumerate}
\item $\pi\colon C\to\PP$ is the projection on the coordinate $x$, which ramifies over $\{0,1-\zeta_3,\sqrt[3]{3},\infty\}$.
\item $\beta_0\coloneqq X^6\in\QQ[X]$, $\Crit(\beta_0)=\{0\}$ so $\beta_0\circ\pi$ ramifies over $\{0,(1-\zeta_3)^6=-3^3,3^2,\infty\}$.
\item $\beta_1\coloneqq\frac{(X+3^3)^3}{3^5(X-3^2)^2}$, $\Crit(\beta_1)=\{0,1\}$ so $\beta=\beta_1\circ\beta_0\circ\pi$ ramifies over $\{0,1,\infty\}$.
\end{enumerate}

The pair $(C,\beta)$ is thus a Belyi pair, and we call $\cD$ the dessin corresponding to $(C,\beta)$. Similarly as in \ref{ex1}, $\cD$ has moduli field $\QQ(\zeta_3,\sqrt[3]{3})$. We will proceed analogously to show that the regular closure $\widetilde{\cD}$ has the same field of moduli. Let us first draw the dessin $\cD_0$ corresponding to the Belyi pair $(\PP,\beta_1\circ\beta_0)$ (see Figure \ref{D0-ex3}), and lift it to the conjugate curves $C^\sigma$ to obtain the conjugate dessins $\cD^\sigma$ for $\sigma\in\Gal(\QQ(\zeta_3,\sqrt[3]{3})/\QQ)$ (see Figure \ref{ex3-dessins}).

\begin{figure}[H]
  \centering
  \includegraphics[width=6cm]{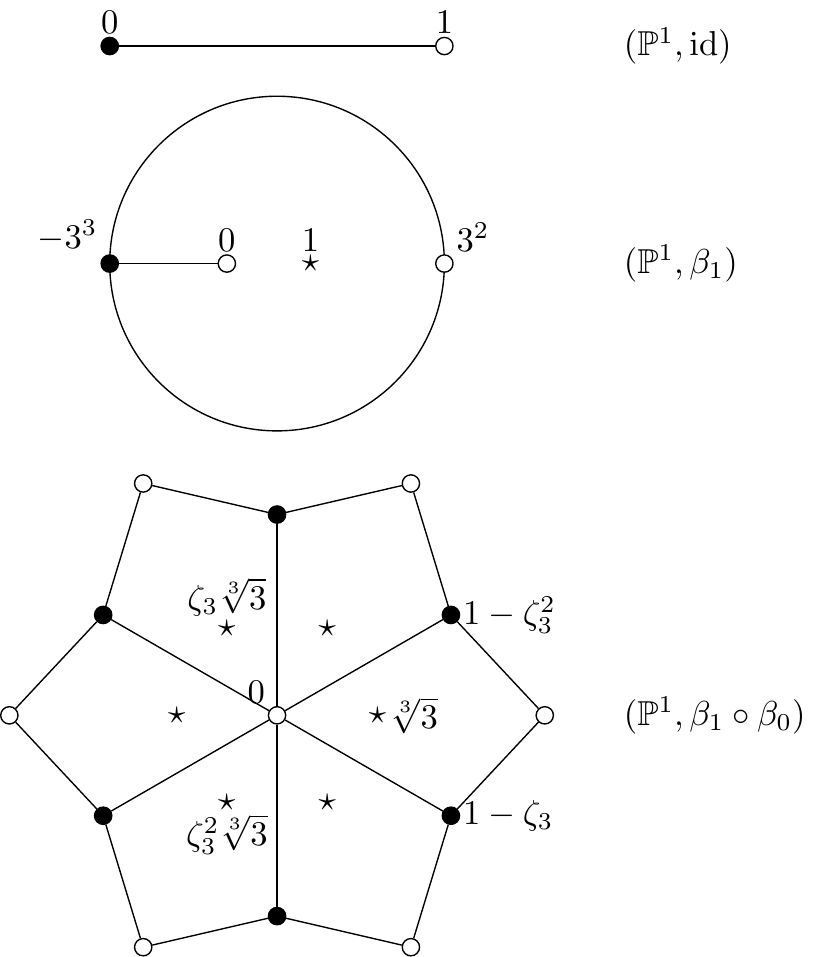}
  \caption{Construction of $\cD_0$.}
  \label{D0-ex3}
\end{figure}

As usual we identify the outermost edges on opposite sides.

\begin{figure}[H]
 \begin{subfigure}{0.5\hsize}
  \begin{center}
    \includegraphics[height=0.9\hsize]{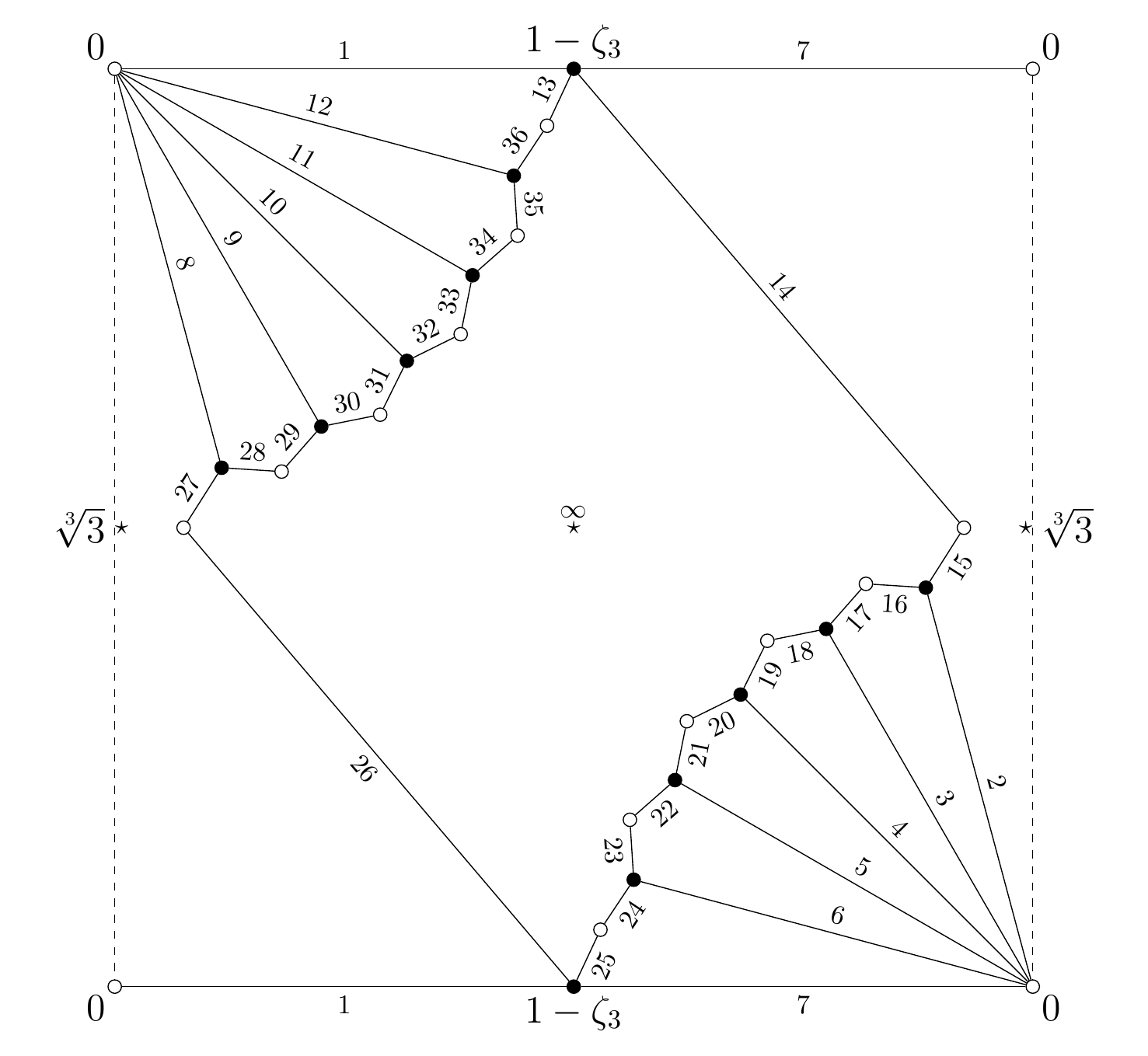}
  \end{center}
  \caption{$\cD_1\coloneqq\cD$}
  \label{ex3-1}
 \end{subfigure}
 \begin{subfigure}{0.5\hsize}
  \begin{center}
    \includegraphics[height=0.9\hsize]{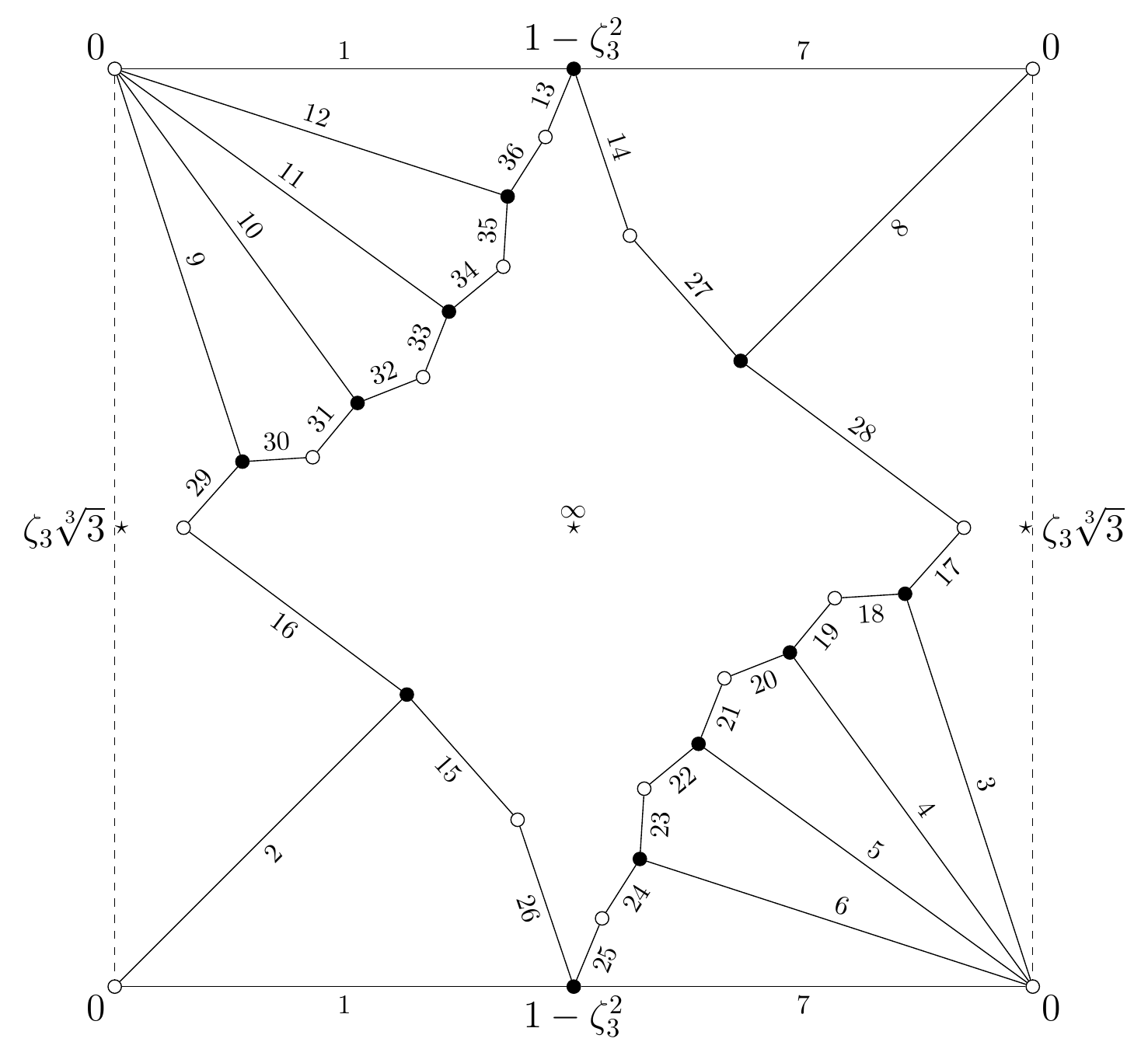}
  \end{center}
  \caption{$\cD_2\coloneqq\cD^\sigma,\sigma\colon(\zeta_3,\sqrt[3]{q})\mapsto(\zeta_3^2,\zeta_3\sqrt[3]{q})$}
 \end{subfigure}
\end{figure}

\begin{figure}[H]\ContinuedFloat
 \begin{subfigure}{0.5\hsize}
  \begin{center}
    \includegraphics[height=0.9\hsize]{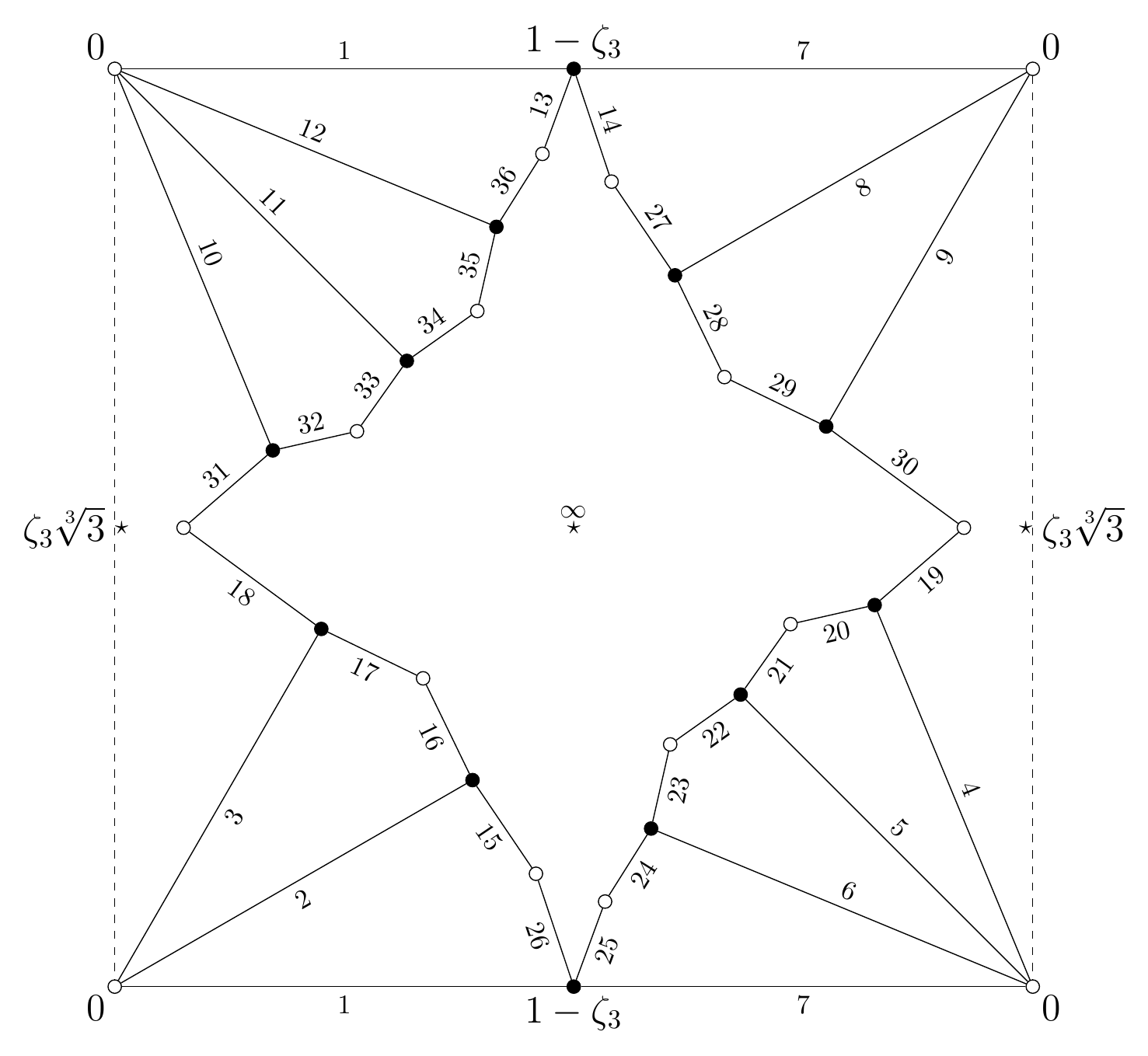}
  \end{center}
  \caption{$\cD_3\coloneqq\cD^\sigma,\sigma\colon(\zeta_3,\sqrt[3]{q})\mapsto(\zeta_3,\zeta_3\sqrt[3]{q})$}
 \end{subfigure}
 \begin{subfigure}{0.5\hsize}
  \begin{center}
    \includegraphics[height=0.9\hsize]{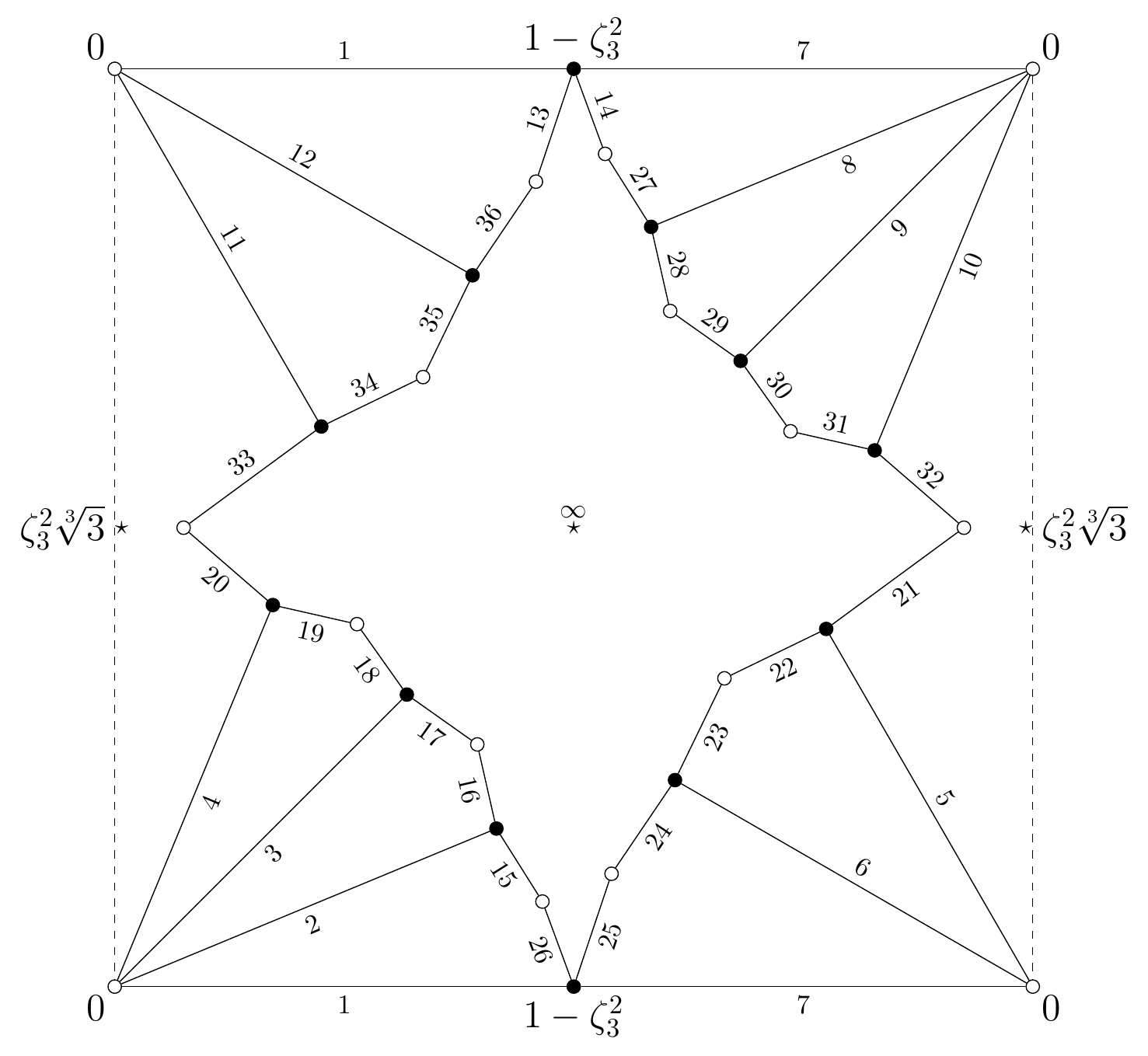}
  \end{center}
  \caption{$\cD_4\coloneqq\cD^\sigma,\sigma\colon(\zeta_3,\sqrt[3]{q})\mapsto(\zeta_3^2,\zeta_3^2\sqrt[3]{q})$}
 \end{subfigure}
\end{figure}

\begin{figure}[H]\ContinuedFloat
 \begin{subfigure}{0.5\hsize}
  \begin{center}
    \includegraphics[height=0.9\hsize]{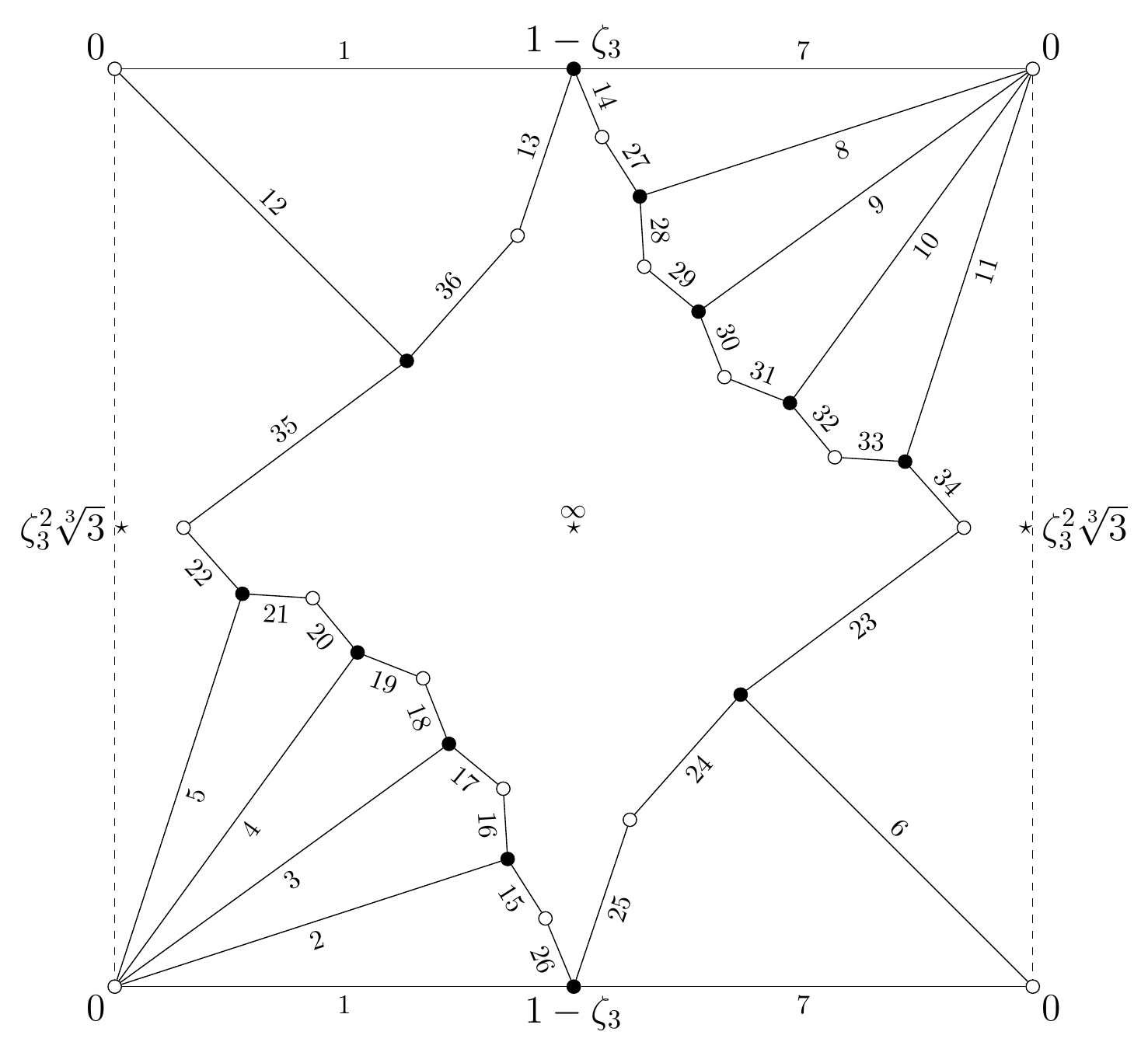}
  \end{center}
  \caption{$\cD_5\coloneqq\cD^\sigma,\sigma\colon(\zeta_3,\sqrt[3]{q})\mapsto(\zeta_3,\zeta_3^2\sqrt[3]{q})$}
 \end{subfigure}
 \begin{subfigure}{0.5\hsize}
  \begin{center}
    \includegraphics[height=0.9\hsize]{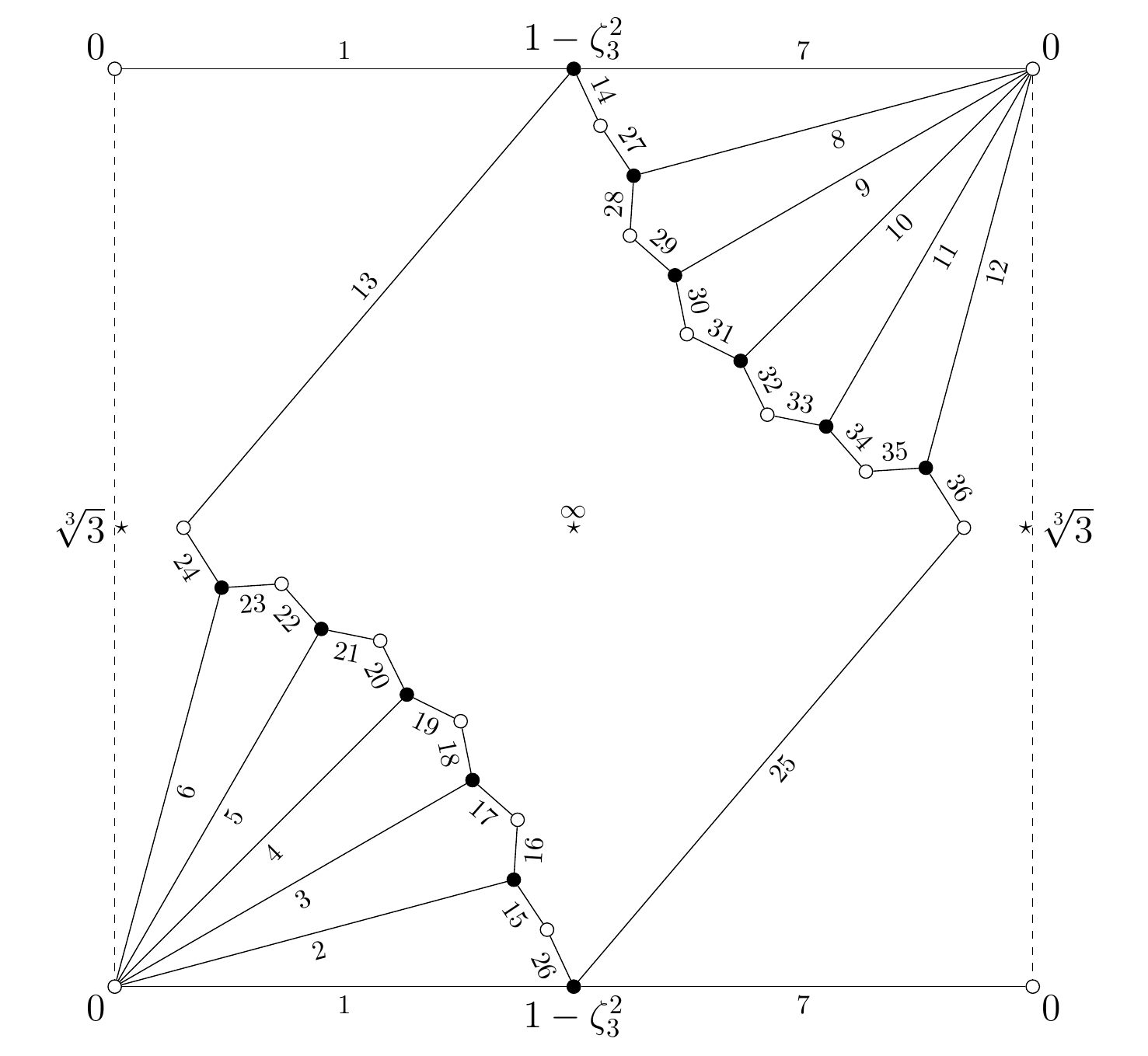}
  \end{center}
  \caption{$\cD_6\coloneqq\cD^\sigma,\sigma\colon(\zeta_3,\sqrt[3]{q})\mapsto(\zeta_3^2,\sqrt[3]{q})$}
  \label{ex3-6}
\end{subfigure}
\caption{Dessins $\cD_1,\dots,\cD_6$ in the Galois orbit of $\cD$.}
\label{ex3-dessins}
\end{figure}

We can now compute the cartographic groups of the dessins. Let $M_k$ denote the monodromy map of $\cD_k$. Then
$$M_k(\genx)=\begin{tabular}{c}
           (1,13,14,7,25,26)(2,15,16)(3,17,18)(4,19,20)(5,21,22)\\
           (6,23,24)(8,27,28)(9,29,30)(10,31,32)(11,33,34)(12,35,36)
         \end{tabular}$$
for all $1\leq k\leq6$, and
\begin{itemize}
\item $M_1(\geny)=\begin{tabular}{c}
                (1,2,3,4,5,6,7,8,9,10,11,12)(13,36)(14,15)(16,17)(18,19)\\
                (20,21)(22,23)(24,25)(26,27)(28,29)(30,31)(32,33)(34,35)
          \end{tabular}$,
\item $M_2(\geny)=\begin{tabular}{c}
                (1,2,3,4,5,6,7,8,9,10,11,12)(13,36)(14,27)(15,26)(16,29)\\
                (17,28)(18,19)(20,21)(22,23)(24,25)(30,31)(32,33)(34,35)
          \end{tabular}$,
\item $M_3(\geny)=\begin{tabular}{c}
                (1,2,3,4,5,6,7,8,9,10,11,12)(13,36)(14,27)(15,26)(16,17)\\
                (18,31)(19,30)(20,21)(22,23)(24,25)(28,29)(32,33)(34,35)
          \end{tabular}$,
\item $M_4(\geny)=\begin{tabular}{c}
                (1,2,3,4,5,6,7,8,9,10,11,12)(13,36)(14,27)(15,26)(16,17)\\
                (18,19)(20,33)(21,32)(22,23)(24,25)(28,29)(30,31)(34,35)
          \end{tabular}$,
\item $M_5(\geny)=\begin{tabular}{c}
                (1,2,3,4,5,6,7,8,9,10,11,12)(13,36)(14,27)(15,26)(16,17)\\
                (18,19)(20,21)(22,35)(23,34)(24,25)(28,29)(30,31)(32,33)
          \end{tabular}$,
\item $M_6(\geny)=\begin{tabular}{c}
                (1,2,3,4,5,6,7,8,9,10,11,12)(13,24)(14,27)(15,26)(16,17)\\
                (18,19)(20,21)(22,23)(25,36)(28,29)(30,31)(32,33)(34,35)
          \end{tabular}$.
\end{itemize}

Using the computer algebra system SageMath \cite{sagemath}, we determined that $$|\langle M_1(\genx),M_1(\geny)\rangle|=42467328=2^{19}\cdot3^4.$$ Moreover, $M_1(\genx)$, $M_1(\geny)$ and $M_1(\genx\geny)$ respectively have orders $6$, $12$ and $12$, so the Euler characteristic of the underlying surface of $\widetilde{\cD_1}$ is $\chi=|\langle M_1(\genx),M_1(\geny)\rangle|\cdot(\frac{1}{\ord M_1(\genx)}+\frac{1}{\ord M_1(\geny)}+\frac{1}{\ord M_1(\genx\geny)}-1)=-28311552=-2^{20}\cdot3^3$, and its genus is $g=1-\frac{\chi}{2}=14155777$.

\medskip
We will now show that $\widetilde{\cD_1}$ is not isomorphic to $\widetilde{\cD_2},\dots,\widetilde{\cD_6}$. We claim that $\omega\coloneqq[\genx^{-1}\geny^2\genx,\genx\geny]\in\ker M_1\setminus\bigcup_{2\leq k\leq6}\ker M_k$, thus concluding the proof. Indeed, we obtain:
\begin{itemize}
\item $M_1(w)=\id$;
\item $M_2(w)=(13,25)(15,27)(21,33)(23,35)$;
\item $M_3(w)=(17,29)(21,33)$;
\item $M_4(w)=(13,25)(15,27)(19,31)(21,33)$;
\item $M_5(w)=(13,25)(17,29)$;
\item $M_6(w)=(13,25)(19,31)(21,33)(23,35)$.
\end{itemize}

We have thus constructed a regular dessin $\widetilde{\cD}$ of degree $2^{19}\cdot3^4$ and genus $14155777$ with moduli field $\QQ(\zeta_3,\sqrt[3]{3})$.

\bibliographystyle{plain}
\bibliography{ref-final}

\end{document}